\documentclass{amsart}

\usepackage{latexsym}
\usepackage{amssymb}
\usepackage{amsmath}
\usepackage{amsthm}
\usepackage{stmaryrd}
\usepackage{hyperref}
\usepackage[utf8]{inputenc}
\usepackage{bbm}
\usepackage{svninfo}

\newcommand{\bC}{\mathbb{C}}
\newcommand{\bN}{\mathbb{N}}
\newcommand{\bR}{\mathbb{R}}

\newcommand{\bE}{\mathbb{E}}
\newcommand{\bP}{\mathbb{P}}

\newcommand{\CA}{\mathcal{A}}
\newcommand{\CB}{\mathcal{B}}

\newcommand{\CK}{\mathcal{K}}

\newcommand{\CL}{\mathcal{L}}
\newcommand{\CM}{\mathcal{M}}
\newcommand{\CN}{\mathcal{N}}
\newcommand{\CR}{\mathcal{R}}

\newcommand{\BB}{\mathbf{B}}
\newcommand{\BK}{\mathbf{K}}

\newcommand{\FM}{\mathfrak{M}}
\newcommand{\FR}{\mathfrak{R}}

\newcommand{\sB}{\mathsf{B}}
\newcommand{\sC}{\mathsf{C}}
\newcommand{\sA}{\mathsf{A}}
\newcommand{\sU}{\mathsf{U}}
\newcommand{\sV}{\mathsf{V}}

\newcommand{\ra}{{\bf{a}}}
\newcommand{\rb}{{\bf{b}}}
\newcommand{\rc}{{\bf{c}}}
\newcommand{\rp}{{\bf{p}}}
\renewcommand{\rq}{{\bf{q}}}
\newcommand{\rr}{{\bf{r}}}
\newcommand{\rA}{{\bf{A}}}

\newcommand{\rCM}{\text{\mathversion{bold}$\CM$}}

\newcommand{\rf}{{\bf{f}}}
\newcommand{\rg}{{\bf{g}}}
\newcommand{\rh}{{\bf{h}}}

\makeatletter

\def\newrefformat#1#2{%
  \@namedef{pr@#1}##1{#2}}
\def\fref#1{\@prettyref#1:}
\def\@prettyref#1:#2:{%
  \expandafter\ifx\csname pr@#1\endcsname\relax%
    \PackageWarning{prettyref}{Reference format #1\space undefined}%
    \ref{#1:#2}%
  \else%
    \csname pr@#1\endcsname{#1:#2}%
  \fi%
}

\newrefformat{eq}{\eqref{#1}}
\newrefformat{cha}{Chapter~\ref{#1}}
\newrefformat{sec}{Section~\ref{#1}}
\newrefformat{apx}{Appendix~\ref{#1}}
\newrefformat{tab}{Table~\ref{#1} on page~\pageref{#1}}
\newrefformat{fig}{Figure~\ref{#1} on page~\pageref{#1}}

\def\indsym#1#2{%
  \setbox0=\hbox{$\m@th#1x$}%
  \kern\wd0%
  \hbox to 0pt{\hss$\m@th#1\mid$\hbox to 0pt{$\m@th#1^{#2}$}\hss}%
  \lower.9\ht0\hbox to 0pt{\hss$\m@th#1\smile$\hss}%
  \kern\wd0}
\newcommand{\ind}[1][]{\mathop{\mathpalette\indsym{#1}}}

\def\nindsym#1#2{%
  \setbox0=\hbox{$\m@th#1x$}%
  \kern\wd0%
  \hbox to 0pt{\hss$\m@th#1\not$\kern1.4\wd0\hss}
  \hbox to 0pt{\hss$\m@th#1\mid$\hbox to 0pt{$\m@th#1^{#2}$}\hss}%
  \lower.9\ht0\hbox to 0pt{\hss$\m@th#1\smile$\hss}%
  \kern\wd0}

\def\dotminussym#1#2{%
  \bbox0=\hbox{$\m@th#1-$}%
  \kern.5\wd0%
  \hbox to 0pt{\hss\hbox{$\m@th#1-$}\hss}%
  \raise.6\ht0\hbox to 0pt{\hss$\m@th#1.$\hss}%
  \kern.5\wd0}

\makeatother

\newcommand{\mynewthm}[3][]{%
  \def\PARAM{#1}
  \ifx\PARAM\empty
  \newtheorem{#2}[theorem]{#3}
  \else
  \newtheorem{#2}{#3}[#1]
  \fi
  \newtheorem*{#2*}{#3}%
  \newrefformat{#2}{#3~\ref{##1}}%
}

\newcommand{\ThmLabel}{Theorem}
\newcommand{\PrpLabel}{Proposition}
\newcommand{\LemLabel}{Lemma}
\newcommand{\FctLabel}{Fact}
\newcommand{\CorLabel}{Corollary}
\newcommand{\DfnLabel}{Definition}
\newcommand{\ConvLabel}{Convention}
\newcommand{\NtnLabel}{Notation}
\newcommand{\CstLabel}{Construction}
\newcommand{\ExmLabel}{Example}
\newcommand{\RmkLabel}{Remark}
\newcommand{\QstLabel}{Question}

\theoremstyle{plain}
\mynewthm{thm}{\ThmLabel}
\mynewthm{prp}{\PrpLabel}
\mynewthm{lem}{\LemLabel}
\mynewthm{fct}{\FctLabel}
\mynewthm{cor}{\CorLabel}

\theoremstyle{definition}
\mynewthm{dfn}{\DfnLabel}
\mynewthm{conv}{\ConvLabel}
\mynewthm{conj}{Conjecture}
\mynewthm{ntn}{\NtnLabel}
\mynewthm{cst}{\CstLabel}

\theoremstyle{remark}
\mynewthm{rmk}{\RmkLabel}
\mynewthm{qst}{\QstLabel}
\mynewthm{exm}{\ExmLabel}

\title{Randomizations of models as metric structures}

\author{Itaï \textsc{Ben Yaacov}}

\address{Itaï \textsc{Ben Yaacov} \\
  Université Claude Bernard -- Lyon 1 \\
  Institut Camille Jordan, CNRS UMR 5208 \\
  43 boulevard du 11 novembre 1918 \\
  69622 Villeurbanne Cedex \\
  France}

\urladdr{\url{http://math.univ-lyon1.fr/~begnac/}}

\thanks{First author supported by
  ANR chaire d'excellence junior THEMODMET (ANR-06-CEXC-007) and
  by the Institut Universitaire de France.}

\author {H.\ Jerome \textsc{Keisler}}

\address{H. Jerome \textsc{Keisler} \\
  Department of Mathematics \\
  University of Wisconsin \\
  480 Lincoln Drive, Madison WI 53706 \\
  U.S.A.}

\urladdr{\url{http://math.wisc.edu/~keisler/}}

\svnInfo $Id: randomise.tex 980 2009-09-23 08:32:59Z begnac $
\thanks{\textit{Revision} {\svnInfoRevision} \textit{of} \today}


\begin{document}

\begin{abstract}
The notion of a randomization of a first order structure was introduced
by Keisler in the paper Randomizing a Model, Advances in Math.\ 1999.
The idea was to form a new structure whose elements are random elements
of the original first order structure.  In this paper we treat
randomizations as continuous structures in the sense of Ben Yaacov and
Usvyatsov.  In this setting, the earlier results show that the
randomization of a complete first order theory is a complete theory
in continuous logic that admits elimination of quantifiers and has
a natural set of axioms.  We show that the randomization operation
preserves the properties of being omega-categorical, omega-stable,
and stable.
\end{abstract}

\maketitle

\section{Introduction}
In this paper we study randomizations of first order structures in the setting
of continuous model theory.  Intuitively, a randomization of a first order
structure $\CM$ is a new structure whose elements are random elements of $\CM$.
In probability theory, one often starts with some structure $\CM$ and studies the
properties of random elements of $\CM$.  In many cases, the random elements of
$\CM$ have properties analogous to those of the original elements of $\CM$.
With this idea in mind, the paper Keisler \cite{Keisler:Randomizing} introduced the notion of a randomization
of a first order theory $T$ as a new many-sorted first order theory.
That approach pre-dated the current development of continuous structures
in  the paper Ben Yaacov and Usvyatsov \cite{BenYaacov-Usvyatsov:CFO}.

Here we formally define a randomization of a first order structure
$\CM$ as a continuous structure in the sense of \cite{BenYaacov-Usvyatsov:CFO}.  This seems to be a
more natural setting for the concept. In this setting,
the results of \cite{Keisler:Randomizing} show that if $T$ is the complete theory of $\CM$,
the theory $T^R$ of randomizations
of $\CM$ is a complete theory in continuous logic which
admits elimination of quantifiers and has a natural set of axioms.

One would expect that the original first order theory $T$
and the randomization theory $T^R$ will have similar model-theoretic
properties.  We show that this is indeed the case for the properties of
$\omega$-categoricity, $\omega$-stability, and stability.
 This provides us with a ready supply of new examples of
 continuous theories with these properties.

In Section 2 we define the randomization theory $T^R$ as a theory in
continuous logic, and restate the results we need from \cite{Keisler:Randomizing} in this
setting.  In  Section 3 we begin with a proof that
a first order theory $T$ is $\omega$-categorical if and only if $T^R$ is
$\omega$-categorical, and then we investigate separable structures.

Section 4 concerns $\omega$-stable theories.
In Subsection 4.1 we prove that a complete theory $T$ is $\omega$-stable if and only if
$T^R$ is $\omega$-stable.  In Subsection 4.2 we extend this result to the case
where $T$ has countably many complete extensions.

Section 5 is about stable theories and independence.  Subsection 5.1 contains
abstract results on fiber products of measures that will be used later.  In
Subsection 5.2 we develop some properties of stable formulas in continuous
theories.  In Subsection 5.3 we prove that a first order theory $T$ is stable
if and only if $T^R$ is stable.  We also give a characterization of independent
types in $T^R$.

In the paper \cite{BenYaacov:RandomVC} another result of this type was proved--- that $T^R$ is
dependent (does not have the independence property) as a continuous theory if and only if $T$ is dependent
as a first order theory.
In this paper we deal only with randomizations of first order structures, but it should
be mentioned that randomizations of  continuous structures
have recently been developed in the papers \cite{BenYaacov:RandomVC} and \cite{BenYaacov:RandomVariables}.

We thank the participants
of the AIM Workshop on the Model Theory of Metric Structures,
held at Palo Alto CA in 2006, and Isaac Goldbring,
for helpful discussions about this work.

\section{Randomizations}  \label{s-rand}

In this section we will restate some notions and results from \cite{Keisler:Randomizing} in the
context of continuous structures.

We will assume that the reader is familiar with continuous model theory as it is
developed in the papers \cite{BenYaacov-Usvyatsov:CFO} and \cite{BenYaacov-Berenstein-Henson-Usvyatsov:NewtonMS}, including the notions of a
structure, pre-structure, signature, theory, and model of a theory.  A \textbf{pre-model}
of a theory is a pre-structure which satisfies each statement in the theory.
 For both first order and continuous logic, we
will abuse notation by treating elements of a structure (or pre-structure) $\CM$ as constant symbols
outside the signature of $\CM$.  These constant symbols will be called \textbf{parameters} from $\CM$.

\textit{We will assume throughout  this paper that $T$ is a consistent first order theory
with  signature $L$ such that each model of $T$  has at least two elements.}


The \textbf{randomization signature} for $L$ is a two-sorted continuous signature $L^R$
with a sort $\BK$ of random elements, and a sort $\BB$ of events.
$L^R$ has an $n$-ary
function symbol $\llbracket\varphi(\cdot)\rrbracket$ of sort $\BK^n\to\BB$
for each first order formula $\varphi$ of $L$ with $n$ free variables,
a $[0,1]$-valued unary predicate symbol $\mu $ of sort $\BB$ for probability,
the Boolean operations $\top, \bot, \sqcup,\sqcap,\neg$ of sort $\BB$,
and distance predicates $d_\BK$ and $d_\BB$ for sorts $\BK, \BB$.
All these symbols are $1$-Lipschitz with respect to each argument.

We will use $\sU,\sV,\ldots$ to denote continuous variables of sort $\BB$.
We will use $x,y,\ldots$ to denote either first order variables or continuous variables
of sort $\BK$, depending on the context.
Given a first order formula $\varphi$ with $n$ free variables,
$\varphi(\overline{x})$ will denote the first order formula formed by replacing
the free variables in $\varphi$ by first order variables $\overline{x}$,
and $\llbracket\varphi(\overline{x})\rrbracket$ will denote the atomic term in $L^R$ formed by
filling the argument places of the function symbol $\llbracket\varphi(\cdot)\rrbracket$
with continuous variables $\overline{x}$ of sort $\BK$.

When working with a structure or pre-structure $(\CK,\CB)$ for $L^R$,
$\sA,\sB,\sC,\ldots$ will denote elements or parameters from $\CB$,
and $\rf,\rg,\ldots$ will denote elements or parameters from $\CK$.
Variables can be replaced by parameters of the same sort.  For example, if
$\overline{\rf}$ is  a tuple of elements of $\CK$,
$\llbracket\varphi(\overline{\rf})\rrbracket$
will be a constant term of $L^R\cup \overline{\rf}$ whose interpretation is an element of $\CB$.

We next define the notion of a randomization $(\CK,\CB)$ of $\CM$.
Informally, the elements $\sB\in\CB$ are events,  that is, measurable subsets of  some probability space $\Omega$,
and the elements $\rf\in\CK$ are random elements of $\CM$,
that is, measurable functions from  $\Omega$ into $\CM$.
By ``measure'' we will always mean ``$\sigma$-additive measure'', unless we
explicitly qualify it as ``finitely additive measure''.

Here is the formal definition.
Given a model $\CM$ of $T$, a \textbf{randomization of} $\CM$ is a pre-structure $(\CK,\CB)$ for $L^R$
equipped with a \textit{finitely additive} measure $\mu$ such that:

\begin{itemize}
\item $(\CB,\mu )$ comes from an atomless finitely
additive probability space $(\Omega,\CB,\mu )$.
\item $\CK$ is a set of functions $\rf\colon\Omega\to\CM$, i.e. $\CK\subseteq\CM^\Omega$.
\item For each formula $\psi(\overline{x})$ of $L$ and tuple
$\overline{\rf}$ in $\CK$, we have
$$ \llbracket\psi(\overline \rf)\rrbracket=\{w\in\Omega:\CM\models\psi(\overline \rf(w))\}.$$
(It follows that the right side belongs to the set of events $\CB$.)
\item For each $\sB\in\CB$ and real $\varepsilon>0$ there are $\rf,\rg\in\CK$
such that $\mu(\sB\triangle\llbracket \rf=\rg\rrbracket)<\varepsilon$,
where $\triangle$ is the Boolean symmetric difference operation.
\item For each formula $\theta(x, \overline{y})$
of $L$, real $\varepsilon > 0$, and tuple $\overline{\rg}$ in $\CK$, there exists $\rf\in\CK$ such that
$$ \mu(\llbracket \theta(\rf,\overline{\rg})\rrbracket\triangle
\llbracket(\exists x\,\theta)(\overline{\rg})\rrbracket)<\varepsilon.$$
\item On $\CK$, the distance predicate $d_\BK$ defines the pseudo-metric
$$d_\BK(\rf,\rg)= \mu \llbracket \rf\neq \rg\rrbracket .$$
\item On $\CB$, the distance predicate $d_\BB$ defines the pseudo-metric
$$d_\BB(\sB,\sC)=\mu ( \sB\triangle \sC).$$

\end{itemize}

Note that the finitely additive measure $\mu$ is determined by the pre-structure
$(\CK,\CB)$ via the distance predicates.

A randomization $(\CK,\CB)$ of $\CM$ is \textbf{full} if in addition

\begin{itemize}
\item
$\CB$ is equal to the set of all events
$ \llbracket\psi(\overline \rf)\rrbracket$
where $\psi(\overline{x})$ is a formula of $L$ and $\overline{\rf}$ is a tuple in $\CK$.
\item $\CK$ is \textbf{full in} $\CM^\Omega$, that is, for each formula $\theta(x, \overline{y})$
of $L$ and tuple $\overline{\rg}$ in $\CK$, there exists $\rf\in\CK$ such that
$$ \llbracket \theta(\rf,\overline{\rg})\rrbracket=\llbracket(\exists x\,\theta)(\overline{\rg})\rrbracket.$$
\end{itemize}

It is shown in \cite{BenYaacov-Usvyatsov:CFO} that each continuous pre-structure induces a unique continuous
structure by identifying elements at distance zero from each other and completing the metrics.
It will be useful to consider these two steps separately here.
By a \textbf{reduced pre-structure} we will mean a pre-structure
such that $d_\BK$ and $d_\BB$ are metrics.  Then every pre-structure $(\CK,\CB)$
induces a unique reduced
pre-structure $(\overline{\CK},\overline{\CB})$ by identifying elements which are at distance zero
from each other.  The induced continuous structure is then obtained by completing the metrics,
and will be denoted by $(\widehat{\CK},\widehat{\CB})$.  We say that a pre-structure
$(\CK,\CB)$ is \textbf{pre-complete} if the reduced pre-structure $(\overline{\CK},\overline{\CB})$
is already a continuous structure, that is,
$(\widehat{\CK},\widehat{\CB})=(\overline{\CK},\overline{\CB})$.
We say that $(\CK,\CB)$ is \textbf{elementarily pre-embeddable} in $(\CK',\CB')$
if $(\widehat{\CK},\widehat{\CB})$ is elementarily embeddable in $(\widehat{\CK}',\widehat{\CB}')$.

Note that for any pre-structure $(\CK,\CB)$, $(\widehat{\CK},\widehat{\CB})$ is
elementarily equivalent to $(\CK,\CB)$, and $(\widehat{\CK},\widehat{\CB})$ is
separable and only if $(\CK,\CB)$ is separable.

In \cite{Keisler:Randomizing}, a randomization of $\CM$ was defined as a three-sorted first order
structure instead of a continuous two-sorted structure, with the value space $[0,1]$
replaced by any first order structure $\CR$ whose theory is an expansion of the theory
of real closed ordered fields which admits quantifier elimination.
When $\CR$ is the ordered field of reals, such a structure can be interpreted as
a randomization in the present sense.

The paper \cite{Keisler:Randomizing} gives axioms for a theory $T^R$, called the randomization theory
of $T$, in three-sorted first order logic.
We now translate these axioms into a theory in
the (two-sorted) continuous logic  $L^R$ in the sense of  \cite{BenYaacov-Usvyatsov:CFO},
with a connective for each continuous function $[0,1]^n\mapsto [0,1]$
which is definable in $\CR$.  We use $\varphi$ for arbitrary formulas of $L$
and $\Phi$ for arbitrary formulas of the continuous logic $L^R$.  Following \cite{BenYaacov-Usvyatsov:CFO},
we use the notation $\forall x(\Phi(x)\leq r)$ for $(\sup_x\Phi(x))\leq r$, and
$\exists x(\Phi(x)\leq r)$ for $(\inf_x\Phi(x))\leq r$.  Thus the existential
quantifiers are understood in the approximate sense.  We also use the notation
$\sU\doteq \sV$ for the statement $d_\BB(\sU,\sV)=0$.

The axioms for the randomization theory $T^R$ of a first order theory $T$ are as follows:
\bigskip

\textit{Validity Axioms}
$$ \forall\overline{x} (\llbracket\psi(\overline{x})\rrbracket\doteq\top)$$
where $\forall\overline{x}\,\psi(\overline{x})$ is logically valid in first order logic.

\textit{Boolean Axioms}  The usual Boolean algebra axioms in sort $\BB$, and the statements
$$ \forall\overline{x}(\llbracket(\neg\varphi)(\overline{x})\rrbracket
\doteq\neg\llbracket\varphi(\overline{x})\rrbracket)$$
$$\forall\overline{x}(\llbracket(\varphi\vee\psi)(\overline{x})\rrbracket
\doteq\llbracket\varphi(\overline{x})\rrbracket\sqcup\llbracket\psi(\overline{x})\rrbracket)$$
$$\forall\overline{x}(\llbracket(\varphi\wedge\psi)(\overline{x})\rrbracket
\doteq\llbracket\varphi(\overline{x})\rrbracket\sqcap\llbracket\psi(\overline{x})\rrbracket)$$

\textit{Distance Axioms}
$$\forall x\forall y\,d_\BK(x,y)= 1- \mu \llbracket x=y\rrbracket,\qquad
\forall \sU\forall \sV\,d_\BB(\sU,\sV)= \mu (\sU\triangle \sV)$$

\textit{Fullness Axioms (or Maximal Principle)}

$$\forall\overline{y}\exists x(\llbracket\varphi(x,\overline{y})\rrbracket
\doteq\llbracket(\exists x\varphi)(\overline{y})\rrbracket)$$
As mentioned previously, the quantifiers should be understood in the approximate sense,
so this axiom can also be written in the equivalent form
$$\sup_{\overline{y}}\inf_x (d_\BB(\llbracket\varphi(x,\overline{y})\rrbracket,
\llbracket(\exists x\varphi)(\overline{y})\rrbracket)) = 0$$

\textit{Event Axiom}
$$ \forall \sU\exists x\exists y(\sU\doteq\llbracket x=y\rrbracket)$$

\textit{Measure Axioms}

$$ \mu [\top]=1\wedge \mu [\bot]=0$$  
$$\forall \sU\forall \sV(\mu [\sU]+\mu [\sV]=\mu [\sU\sqcup \sV]+\mu [\sU\sqcap \sV])$$

\textit{Atomless Axiom}
$$\forall \sU \exists \sV(\mu [\sU\sqcap \sV]=\mu [\sU]/2)$$
\medskip

\textit{Transfer Axioms}

$$ \llbracket\varphi\rrbracket \doteq\top$$
where $\varphi\in T$.
\bigskip


\begin{thm} \label{complete}
 (\cite{Keisler:Randomizing}, Theorem 3.10, restated).  If $T$ is complete, then  $T^R$ is complete.
\end{thm}

\begin{prp}  (\cite{Keisler:Randomizing}, Proposition 4.3, restated)
Let $T$ be the complete theory of $\CM$.
Every randomization $(\CK,\CB)$ of $\CM$ is a pre-model of the randomization theory
$T^R$.
\end{prp}

Let $(\CK,\CB)$ and $(\CK',\CB')$ be pre-models of $T^R$.
We say that $(\CK,\CB)$ \textbf{represents}
$(\CK',\CB')$ if their corresponding reduced pre-structures are isomorphic.


\begin{thm} \label{firstrepresentation}
(First Representation Theorem)
Let  $T$ be the complete theory of $\CM$.
Every pre-model of $T^R$  is represented by some
randomization of $\CM$.
\end{thm}

\begin{proof}  Let $(\CK,\CB)$ be a pre-model of $T^R$.
It follows from \cite{Keisler:Randomizing}, Corollary 6.6 and Theorem 5.7, that the reduced structure
$(\overline{\CK},\overline{\CB})$ of $(\CK,\CB)$ is isomorphically embeddable in the reduced pre-structure
of a randomization of $\CM$.  Since $(\CK,\CB)$ satisfies
the Fullness and Event Axioms, $(\overline{\CK},\overline{\CB})$ is isomorphic to
the reduced pre-structure of a randomization of $\CM$.
\end{proof}

\begin{dfn}  A pre-model $(\CK,\CB)$ of $T^R$ has \textbf{perfect witnesses} if
the existential quantifiers in the Fullness Axioms and the Event Axiom
have witnesses in which the axioms hold exactly rather than merely approximately.
That is,

\textit{Fullness}: For each $\overline{\rg}$ in $\CK^n$ there exists $\rf\in\CK$ such that
$$ \llbracket\varphi(\rf,\overline{\rg})\rrbracket \doteq
\llbracket(\exists x\varphi)(\overline{\rg})\rrbracket.$$

\textit{Event}:  For each $\sB\in\CB$ there exist $\rf,\rg\in\CK$ such that
$\sB\doteq\llbracket \rf=\rg\rrbracket$.
\end{dfn}

The first order models considered in \cite{Keisler:Randomizing} are pre-models
with perfect witnesses when viewed as metric structures.

\begin{prp}  \label{p-perfect}  (\cite{Keisler:Randomizing}, Proposition 4.3, restated)
Let $T$ be the complete theory of $\CM$.
Every full randomization $(\CK,\CB)$ of $\CM$ has perfect witnesses.
\end{prp}

\begin{thm} \label{representation}
(Second Representation Theorem) (\cite{Keisler:Randomizing}, Theorem 4.5, restated)
Let  $T$ be the complete theory of $\CM$.
Every pre-model of $T^R$ with perfect witnesses is represented by some
full randomization of $\CM$.
\end{thm}

We now show that every model of $T^R$ has perfect witnesses.

\begin{thm} \label{t-perfect}
 For any first order theory $T$, every pre-complete model
of $T^R$ has perfect witnesses.  In particular, every model of $T^R$
has perfect witnesses.
\end{thm}

\begin{proof}  Let $(\CK,\CB)$ be a pre-complete model of $T^R$.
We first show that $(\CK,\CB)$ has perfect witnesses for the Fullness Axioms.
Consider a first order formula $\varphi(x,\overline{y})$, and let
$\overline{\rg}$ be a tuple in $\CK$.  Then for each $n\in\omega$ there exists
$\rf_n\in\CK$ such that
$$d_\BB(\llbracket\varphi(\rf_n,\overline{\rg})\rrbracket,
\llbracket(\exists x\varphi)(\overline{\rg})\rrbracket) < 2^{-n}.$$
Using the Fullness Axioms, we can get a sequence $\rh_n$ in $\CK$
such that $\rh_1=\rf_1$ and with probability at least $1-2^{-2n}$,
$\rh_{n+1}$ agrees with $\rh_n$ when $\varphi(\rh_n,\overline{\rg})$
holds, and agrees with $\rf_{n+1}$ otherwise.  Then
$$d_\BB(\llbracket\varphi(\rh_n,\overline{\rg})\rrbracket,
\llbracket(\exists x\varphi)(\overline{\rg})\rrbracket) < 2^{-n}+2^{-2n}$$
and
$$d_\BK(\rh_n,\rh_{n+1})<2^{-n}+2^{-2n}.$$
Then the sequence $\rh_n$ is Cauchy convergent with respect to $d_\BK$.
By pre-completeness, $\rh_n$ converges to an element $\rh\in\CK$ with respect
to $d_\BK$. It follows that
$$d_\BB(\llbracket\varphi(\rh,\overline{\rg})\rrbracket,
\llbracket(\exists x\varphi)(\overline{\rg})\rrbracket) = 0,$$
as required.

It remains to show that $(\CK,\CB)$ has perfect witnesses for the Event Axiom.
Let $\sB\in\CB$.  By the Event Axiom, for each $n\in\omega$ there exists
$\rf_n, \rg_n\in\CK$ such that $\mu[\sB\triangle\llbracket \rf_n=\rg_n\rrbracket]<2^{-n}$.
Since every model of $T$ has at least two elements, there must exist
$\rf',\rg'\in\CK$ such that $\llbracket \rf'=\rg'\rrbracket\doteq\bot$.  Taking perfect witnesses for
 the Fullness Axioms, we can obtain elements
$\rh_n\in\CK$ such that $\rh_n$ agrees with $\rf'$ on
$\llbracket \rf_n=\rg_n\rrbracket$, and $\rh_n$ agrees with $\rg'$ on
$\llbracket \rf_n\neq \rg_n\rrbracket$.  Then the sequence $\rh_n$ is Cauchy convergent
with respect to $d_\BK$, and hence converges to some $\rh$ in $\CK$.
It follows that $\sB\doteq \llbracket \rf'=\rh\rrbracket$.
\end{proof}

\begin{cor}  \label{c-represent}
Let  $T$ be the complete theory of $\CM$.
Every pre-complete model of $T^R$, and hence every model of $T^R$,
 is represented by some full randomization of $\CM$.
\end{cor}

\begin{thm}  \label{qe} (\cite{Keisler:Randomizing} Theorems 3.6 and 5.1, restated).
For any first order theory $T$,
the randomization theory $T^R$ for $T$ admits
\emph{strong} quantifier elimination.
\end{thm}

This means that every formula $\Phi$ in the continuous language $L^R$
is $T^R$-equivalent to a formula with the same free variables
and no quantifiers of sort $\BK$ or $\BB$
(whereas ordinary quantifier elimination for continuous logic means
that every formula can be arbitrarily well approximated by quantifier-free formulas).
  Note that for each first order formula $\varphi(\overline{x})$ of
$L$, $\mu\llbracket\varphi(\overline{x})\rrbracket$ is an atomic formula of $L^R$ which has no quantifiers.
The first order quantifiers within $\varphi(\overline{x})$ do not count
as quantifiers in $L^R$.

By the Event Axiom and Theorem \ref{t-perfect}, in a model of $T^R$,
any element of sort $\BB$ is equal to a term $\llbracket \rg_1=\rg_2\rrbracket$
with parameters $\rg_1,\rg_2$ of sort $\BK$.  Therefore, in all discussions of types in
the theory $T^R$, we may confine our attention
to types of sort $\BK$ over parameters of sort $\BK$.

The space of first order $n$-types in $T$
will be denoted by $S_n(T)$.  If $T=Th(\CM)$ and $A$ is a set of parameters in $\CM$,
$S_n(T(A))$ is the space of first order $n$-types
in $T$ with parameters in $A$. The space of continuous $n$-types
in $T^R$ with variables of sort $\BK$ will be denoted by $S_n(T^R)$.
We will use boldface letters
$\rp,\rq,\ldots$ for types of sort $\BK$, and boldface $\rA$ for a set of parameters of sort $\BK$.
If $(\CK,\CB)$ is a pre-model of $T^R$ and $\rA$ is a set of parameters in $\CK$,
$S_n(T^R(\rA))$ is the space of continuous $n$-types in $T^R$ with variables
of sort $\BK$ and parameters from $\rA$.

Recall from \cite{BenYaacov-Usvyatsov:CFO} that  for each $\rp\in S_n(T^R)$ and formula $\Phi(\overline{x})$ of $L^R$,
we have $(\Phi(\overline{x}))^\rp\in[0,1]$.

Let $\FR(S_n(T))$ be the space of regular Borel probability measures on $S_n(T)$.
The next corollary follows from quantifier elimination and the axioms of $T^R$.

\begin{cor}  \label{c-types}
For every  $\rp\in S_n(T^R)$ there is a unique
measure $\nu_\rp\in\FR(S_n(T))$
such that for each formula $\varphi(\overline{x})$ of $L$,
$$\nu_\rp(\{q:\varphi(\overline{x})\in q\})=(\mu\llbracket\varphi(\overline{x})\rrbracket)^\rp.$$
Moreover, for each measure $\nu\in\FR(S_n(T))$
there is a unique $\rp\in S_n(T^R)$ such that $\nu=\nu_\rp$.

Similarly for types with infinitely many variables.
\end{cor}

We may therefore identify the type space $S_n(T^R)$ with the space $\FR(S_n(T))$.

\begin{rmk} \label{measure-algebra}
 In the special case that $\CM$ is the trivial two-element structure with
only the identity relation, the above  results show that
the continuous theory of atomless measure algebras is complete and admits
quantifier elimination.  (See \cite{Keisler:Randomizing}, Section 7A).
This fact was given a direct proof in \cite{BenYaacov-Usvyatsov:CFO}.
\end{rmk}

\begin{rmk}  Throughout this paper, we could have worked with a one-sorted
randomization theory with only the sort $\BK$ instead of a two-sorted theory
with sorts $\BK$ and $\BB$.  In this formulation, we would use the results
of Section 7C of \cite{Keisler:Randomizing}, where the sort $\BB$ is eliminated.
This approach will be taken at the end of Section \ref{ss-main}.
The main advantages of the event sort $\BB$
are that it allows a nicer set
of axioms for $T^R$, and makes it easier to describe the models of $T^R$.
\end{rmk}

\section{Separable Structures}

In this section we consider small randomizations of $\CM$.
In order to cover the case that $T$ has  finite models, we adopt the
following convention:

\begin{center}
\textit{``$\CM$ is countable'' means that
the cardinality of $\CM$ is finite or $\omega$.}
\end{center}

\noindent So a complete first order theory $T$ is
$\omega$-categorical if and only if it either has a finite model or has a
unique model of cardinality $\omega$ up to isomorphism.

We first show that the randomization operation preserves $\omega$-categoricity.
 We will use
the following necessary and sufficient condition for a continuous theory to
be $\omega$-categorical.


\begin{fct} \label{t-ryll-nardz} (\cite{BenYaacov-Berenstein-Henson-Usvyatsov:NewtonMS}, Theorem 13.8).
Let $U$ be a complete continuous theory with a countable signature.
Then $U$ is $\omega$-categorical if and only if for each $n\geq 1$,
every type in $S_n(U)$ is realized in every model of $U$.
\end{fct}

\begin{thm}  \label{t-omega-cat}
Suppose $T$ has a countable signature.  Then
$T$ is $\omega$-categorical if and only if $T^R$ is $\omega$-categorical.
\end{thm}

\begin{proof}
Suppose that $T$ is not $\omega$-categorical.  Then by the Ryll-Nardzewski Theorem,
for some $n$ there is an $n$-type $q\in S_n(T)$ which is omitted in some
countable model $\CM$ of $T$.   Then the type $\rp\in S_n(T^R)$
such that $\nu_\rp(\{q\})=1$ is  omitted in the separable pre-complete model $(\CM^{[0,1]},\CL)$,
of $T^R$, so $T^R$ is not $\omega$-categorical.

Now suppose $T$ is $\omega$-categorical.  Then $T=Th(\CM)$ for some $\CM$.
By Ryll-Nardzewski's theorem, for each $n$ the type space $S_n(T)$ is finite.
Let $\rp\in S_n(T^R)$.  By Corollary \ref{c-types}, $\nu_\rp\in\FR(S_n(T))$, and
for each formula $\varphi(\overline{x})$ of $L$,
$$\nu_\rp(\{q:\varphi(\overline{x})\in q\})=(\mu\llbracket\varphi(\overline{x})\rrbracket)^\rp.$$
Since $S_n(T)$ is finite, for each $q\in S_n(T)$ there is a formula
$\varphi_q(\overline{x})$ of $L$ such that $q$ is the set of $T$-consequences of $\varphi_q$.
Let $(\CK',\CB')$ be a model of $T^R$.  By Corollary \ref{c-represent}, $(\CK',\CB')$
is represented by some full pre-complete randomization $(\CK,\CB)$ of $\CM$.
By Fact \ref{t-ryll-nardz}, it suffices to show that $\rp$ is realized in $(\CK,\CB)$.
By the axioms and pre-completeness, there are events $\sB_q,  q\in S_n(T)$ in $\CB$
such that $\sB_q, q\in S_n(T)$ partitions $\Omega$ and $\mu(\sB_q)=\nu_\rp(\{q\})$.
Since $(\CK,\CB)$ has perfect witnesses, for each $q$ there is a tuple
$\overline{\rf}_q$ in $\CK$ such that
$\mu\llbracket\varphi_q(\overline{\rf}_q)\rrbracket=1$.
Using the Fullness and Event Axioms, there is a tuple $\overline{\rf}$ in $\CK$
which agrees with $\overline{\rf}_q$ on $\sB_q$ for each $q\in S_n(T)$.
It follows that $\overline{\rf}$ realizes $\rp$ in $(\CK,\CB)$, as required.
\end{proof}
\medskip

We next give a method for constructing small models of $T^R$,
and then introduce the notion of a strongly separable model of $T^R$.

\begin{dfn}
Let $(\Omega,\CB,\mu)$ be an atomless finitely additive probability space and let
$\CM$ be a structure for $L$.

A \textbf{$\CB$-deterministic element} of $\CM$ is a constant function in $\CM^\Omega$
(which may be thought of as an element of $\CM$).

A \textbf{$\CB$-simple random element} of $\CM$ is a $\CB$-measurable
function in $\CM^\Omega$ with finite range.

A \textbf{$\CB$-countable random element} of $\CM$ is a $\CB$-measurable
function in $\CM^\Omega$ with countable range .
\end{dfn}

\begin{exm}  \label{examples}
(\cite{Keisler:Randomizing}, Examples 4.6 and 4.11))

(i) The set  $\CK_S$ of $\CB$-simple random elements of
$\CM$ is full, and $(\CK_S,\CB)$ is a full randomization of $\CM$.

(ii)  The set  $\CK_C$ of $\CB$-countable random elements of
$\CM$ is full.  If $(\Omega,\CB,\mu)$ is a probability space (i.e. is $\sigma$-additive) then
$(\CK_C,\CB)$ is a full pre-complete randomization of $\CM$.
\end{exm}

\textit{We assume for the rest of this section that the signature $L$ of $T$ is countable}
\medskip

Note that when $L$ is countable, every
Borel probability measure on $S_n(T)$ is regular.
(This follows from \cite{Halmos:MeasureTheory}, page 228, and the fact that
every open set in $S_n(T)$ is a countable union of compact sets).

\begin{dfn} Let  $([0,1],\CL,\lambda)$
be the natural atomless Borel probability measure on $[0,1]$.
Given a model $\CM$ of $T$, let $(\CM^{[0,1]},\CL)$ be the pre-structure
whose universe of sort $\BK$ is
the set of $\CL$-countable random elements of $\CM$.
A pre-model $(\CK,\CB)$ of $T^R$ is \textbf{strongly separable} if
$(\CK,\CB)$ is elementarily pre-embeddable in
$(\CM^{[0,1]},\CL)$ for some countable model $\CM$ of $T$.
\end{dfn}

Note that $(\CM^{[0,1]},\CL)$ is only a pre-structure,
not a reduced pre-structure, because $\CL$ has nonempty null sets.

\begin{cor} \label{c-precomplete}
If $\CM$ is a countable model of $T$, then $(\CM^{[0,1]},\CL)$ is
a full randomization of $\CM$ and is a strongly separable pre-complete model of $T^R$.
\end{cor}

\begin{proof} By Example \ref{examples} (ii).
\end{proof}

\begin{lem}  \label{simple}
Consider a first order $L$-structure $\CM$.
Suppose $(\Omega,\CB,\mu)$ is a probability space and
$\CA$ is a subalgebra of $\CB$ which is dense with respect to the distance predicate
$d_\BB(\sU,\sV)=\mu(\sU\triangle \sV)$.
Then the set $\CK_S$ of $\CA$-simple random elements of $\CM$ is  dense in the set $\CK$
of $\CB$-countable random elements of $\CM$ with respect to the distance predicate $d_\BK$.
Therefore $(\CK_S,\CA)$ and $(\CK,\CB)$
induce the same continuous structure, and $\widehat{\CK}_S=\widehat{\CK}$.
\end{lem}

\begin{proof}  By Example \ref{examples}, $(\CK_S,\CA)$ and $(\CK,\CB)$ are full randomizations of $\CM$.
Let $\rf$ be an element of $\CK$ with range $\{ a_n: n\in\mathbb{N}\}$.
Let $\varepsilon>0$.  Take $n$ such that
$$\mu\left\llbracket\bigvee_{m<n}(\rf=a_m)\right\rrbracket>1-\varepsilon/2.$$
For each $m<n$, let
$$ \sB_m=\{w:\rf(w)=a_m\}\in\CB.$$
Then the sets $\sB_m$ are disjoint and
$$ \mu\left[\bigcup_{m<n} \sB_m\right]> 1-\varepsilon/2.$$
For each $m<n$, there is a set $\sA_m\in\CA$ such that $d_\BB(\sA_m,\sB_m)<\varepsilon/(4n^2)$.
Let $\sC_m=\sA_m\setminus\bigcup_{k<m}\sA_k$. Then $\sC_m\in\CA$, the sets $\sC_m$
are disjoint, and one can check that $d_\BB(\sC_m,\sB_m)<\varepsilon/(2n)$.
There is an $\CA$-simple $\rg\in\CK_S$ such that $\rg(w)=a_m$ whenever $m<n$ and $w\in \sC_m$.
Then $\rf(w)=\rg(w)$ whenever $w\in \sB_m\cap \sC_m$, so
$$\mu\llbracket \rf= \rg\rrbracket>1-\varepsilon,$$
and hence
 $d_\BK(\rf,\rg)<\varepsilon$.  This shows that $\CK_S$ is dense in $\CK$.
\end{proof}

\begin{cor}  \label{c-separable}
Every strongly separable pre-model of $T^R$ is separable.
\end{cor}

\begin{proof}  Let $(\CK,\CB)$ be a strongly separable pre-model of $T^R$, so that
$(\CK,\CB)$ is elementarily pre-embeddable in $(\CM^{[0,1]},\CL)$ for some countable
model $\CM$ of $T$.
$\CL$ has a countable dense subalgebra $\CA$.  By Lemma \ref{simple},
the set of $\CA$-simple random elements of $\CM$ is a countable dense subset
of $\CM^{[0,1]}$ with respect to $d_\BK$.  Therefore $(\CM^{[0,1]},\CL)$
is separable, and hence $(\CK,\CB)$ is separable.
\end{proof}

\begin{lem}  \label{l-strongsep}  Let $T$ be complete, $\rp\in S_n(T^R)$, and  $\nu_\rp$ be
the corresponding measure on $S_n(T)$ defined in Corollary \ref{c-types}.
Then $\rp$ is realized in
some strongly separable model of $T^R$
if and only if there is a countable set $C\subseteq S_n(T)$
such that $\nu_\rp(C)=1$.
\end{lem}

\begin{proof}  Suppose $\rp$ is realized in
some strongly separable model of $T^R$.
Then $\rp$ is realized by an $n$-tuple $\overline{\rg}$ in the pre-complete model
$(\CM^{[0,1]},\CL)$ of $T^R$ for some countable model $\CM$ of $T$.
Hence $\nu_\rp(C)=1$ where $C$ is the set of types of elements of the range of $\overline{\rg}$.

Suppose  $\nu_\rp(C)=1$ for some countable set $C\subseteq S_n(T)$, and let $\CM$ be a countable
model of $T$ which realizes each $q\in C$.
Then $\rp$ is realized in the strongly separable pre-model $(\CM^{[0,1]},\CL)$ of $T^R$
by any tuple  $\overline{\rg}$ such that
$$\lambda(\{r:tp^{\CM}(\overline{\rg}(r))=q\})=\nu_\rp(\{q\})$$
for each $q\in C$.
\end{proof}


\begin{exm}  (i) Let $T$ be the first order theory with countably many
independent unary relations $P_0,P_1,\ldots$, and let $\rp$ be the type in $S_1(T^R)$
such that the events $\llbracket P_n(x)\rrbracket, n\in\mathbb{N}$ are independent
with respect to $\rp$.  Then $\rp$ is not realized in a strongly separable model of
$T^R$, and hence $T^R$ has separable models which are not strongly separable.

(ii)  Let $\bR$ be the ordered field of real numbers, as a first order rather than
continuous structure.  Let $\CK$ be the set of all
Lebesgue measurable functions from $[0,1]$ into $\bR$.  By Proposition 4.12 in \cite{Keisler:Randomizing},
 $(\CK,\CL)$ is a full pre-complete randomization of $\bR$.   $(\CK,\CL)$ is not separable, and the
 type of the identity function on $[0,1]$ in $(\CK,\CL)$ is not realized in
a strongly separable model.
\end{exm}


\begin{lem} \label{l-omega-sat}
Let $\CM$ be countable.  Then $\CM$ is $\omega$-saturated if and only if $(\CM^{[0,1]},\CL)$
is $\omega$-saturated as a continuous pre-structure.
\end{lem}

\begin{proof}
Suppose first that $\CM$ is not $\omega$-saturated.  Then there is a tuple $\overline{a}$ in $\CM$,
a countable elementary extension $\CN$ of $\CM$, and an element $b\in\CN$
such that the type of $b$ in $(\CN,\overline{a})$ is not realized in $(\CM,\overline{a})$.
By considering the types of the constant functions from $[0,1]$ to $\overline a$ and $b$, we see that
$\CM^{[0,1]}$ is not $\omega$-saturated.

Now suppose that $\CM$ is $\omega$-saturated.
Let $T$ be the complete theory of $\CM$ and let $\CK=\CM^{[0,1]}$.
It suffices to show that for each finite tuple $\overline{\rg}$ in $\CK$,
every type $\rp\in S_1(T^R(\overline{\rg}))$ is realized in $(\CK,\CL)$.
Let $\{\overline{b}_n:n<\omega\}$ be the range of $\overline{\rg}$ in $\CM$ and
let
$$\sB_n=\{r\in [0,1]:\overline{\rg}(r)=\overline{b}_n\}=\llbracket\overline{\rg}=\overline{b}_n\rrbracket.$$
Note that $\sB_n, n\in\omega$ is a partition of $[0,1]$ in $\CL$, and
$$\sum_n \mu[\sB_n]=1.$$  Since $\CM$
is  countable and $\omega$-saturated, for each $n$ there are countably many types
$q\in S_1(T(\overline{b}_n))$, and each of these types is realized by an element
$a_q$ in $\CM$.  Let $\rp\in S_1(T^R(\overline{\rg}))$.
Let $\rA$ be the set of all deterministic elements of $\CK$, and extend $\rp$
to a type $\rp'$ in \hbox{$S_1(T^R(\rA\cup\overline{\rg}))$}.
For each $n$ and first order formula $\varphi(x,\overline{b}_n)$,
$\mu[\llbracket\varphi(x,\overline{b}_n)\rrbracket\sqcap \sB_n]$
is a continuous formula with parameters in $\rA\cup \overline{\rg}$.  Let
$$\alpha(\varphi,n)=(\mu[\llbracket\varphi(x,\overline{b}_n)\rrbracket\sqcap \sB_n])^{\rp'}$$
and for each $q\in S_1(T(\overline{b}_n))$ let
$$\beta(q,n)=\inf\{\alpha(\varphi,n):\varphi\in q\}.$$
Since $S_1(T(\overline{b}_n))$ is countable for each $n$, we have
$$ \sum\{\beta(q,n):q\in S_1(T(\overline{b}_n))\}=\mu[\sB_n].$$
There is a function $\rf\in\CK$ such that for each $n$ and $q\in S_1(T(\overline{b}_n))$,
$$ \mu\{r\in \sB_n: \rf(r)=a_q\}=\beta(q,n).$$
Then for each $n$ and first order formula $\varphi(x,b_n)$,
$$\mu[\llbracket \varphi(\rf,\overline{\rg})\rrbracket\sqcap \sB_n]\geq \alpha(\varphi,n).$$
It follows that $\rf$ realizes $\rp$ in $(\CK,\CL)$.
\end{proof}

\begin{thm}  \label{t-omega-sat}
Let $T$ be complete.  The following are equivalent.

(i)  $T$ has a countable $\omega$-saturated model.

(ii)  $T^R$ has a separable $\omega$-saturated model.

(iii)  Every separable model of $T^R$ is strongly separable.
\end{thm}

\begin{proof}  (i) implies (ii) and (iii): Let $\CM$ be a countable  $\omega$-saturated
model of $T$ and let $\CK=\CM^{[0,1]}$.
Then $(\CK,\CL)$ is strongly separable, and hence is separable by Corollary \ref{c-separable}.
By Lemma \ref{l-omega-sat}, $(\CK,\CL)$ is $\omega$-saturated, so (ii) holds.  It follows that
 every separable model of
$T^R$ is elementarily pre-embeddable in the strongly separable pre-model $(\CK,\CL)$,
and hence is itself strongly separable.  This proves (iii).


(ii) implies (i):    Assume that (i) fails.  Then for some $n$, $S_n(T)$ is uncountable.
For each $q\in S_n(T)$ there is an $n$-type $\rq'$ in $S_n(T^R)$
such that $\nu_{\rq'}(\{q\})=1$.   Let $(\CK,\CB)$ be an $\omega$-saturated
model of $T^R$.
Each  $n$-type $\rq'$ is realized by an $n$-tuple $\overline{\rf}_q$ in $(\CK,\CB)$.
Suppose $p,q\in S_n(T)$ and $p\ne q$.  Then $\nu_{\rp'}(\{q\})=0$ and $\nu_{\rq'}(\{q\})=1$.
By the First Representation Theorem \ref{firstrepresentation}, $(\CK,\CB)$ is represented by a
randomization of a model of $T$, so we may assume that $(\CK,\CB)$ is already a randomization
of a model of $T$.  It follows that
$\overline{\rf}_p\ne \overline{\rf}_q$ almost everywhere, and
$\mu\llbracket \overline{\rf}_p\ne\overline{\rf}_q\rrbracket=1$.
Let $d_n$ be the metric on $\CK^n$ formed by adding the $d_\BK$ distances at each coordinate.
Then
$$d_n(\overline{\rf}_p,\overline{\rf}_q)\ge
\mu\llbracket \overline{\rf}_p\ne\overline{\rf}_q\rrbracket=1.$$
  Since $S_n(T)$ is uncountable, we see
that the metric space $(\CK^n,d_n)$ is not separable.  Therefore
$(\CK,d_\BK)$ is not separable,  so $(\CK,\CB)$ is not
separable and (ii) fails.

(iii) implies (i).  Again assume (i) fails and $S_n(T)$ is
uncountable.  Then by enumerating the formulas of $L$ one can construct a
measure $\nu\in\FR(S_n(T))$ such that
$\nu(\{q\})=0$ for each $q\in S_n(T)$.  By Corollary \ref{c-types} we may take $\rp\in S_n(T^R)$
such that $\nu_\rp=\nu$.  By Lemma \ref{l-strongsep},
$\rp$ cannot be realized in a strongly separable model, but $\rp$ can be realized in a
separable model.  Therefore (iii) fails.
\end{proof}
\medskip

\section{$\omega$-Stable Theories}

\textit{In this section we continue to assume that the signature $L$ of $T$ is countable.}
\medskip

As explained in Section \ref{s-rand}, when considering types in $T^R$, we may confine our attention
to types of sort $\BK$ over parameters of sort $\BK$.
For each model $(\CK,\CB)$ of $T^R$ and set of parameters $\rA\subseteq\CK$, the $d$ metric on the type space
$S_1(T^R(\rA))$ is defined by
$$ d(\rp,\rq)=\inf\{d_\BK(\rf,\rg)\colon tp(\rf/\rA)=\rp \mbox{ and }
tp(\rg/\rA)=\rq \mbox{ in some } (\CK',\CB')\succ (\CK,\CB)\}.$$
We say that the space $S_1(T^R(\rA))$ is \textbf{separable} if it is separable with
respect to the $d$ metric.
Following \cite{BenYaacov-Berenstein-Henson-Usvyatsov:NewtonMS},  we say that $T^R$ is \textbf{$\omega$-stable} if $S_1(T^R(\rA))$ is separable
for every model $(\CK,\CB)$ of $T^R$ and countable set $\rA$ of parameters in $\CK$.

We show that if $T$ has at most countably many complete extensions,
then the randomization operation preserves $\omega$-stability.
We first take up the case that $T$ is complete.

\subsection{Complete Theories}

\begin{thm}  \label{omega-stable}  A complete theory
$T$ is $\omega$-stable if and only if $T^R$ is $\omega$-stable.
\end{thm}

\begin{proof}  Assume $T^R$ is $\omega$-stable. Let $\CM$ be an arbitrary countable
model of $T$, let $\CK=\CM^{[0,1]}$, and let $\rA$ be the set of deterministic
elements of $\CK$.  Then $\rA$ is countable, so $S_1(T^R(\rA))$ is separable.

For each $q\in S_1(T(\CM))$ let $\rq'$ be the type in $S_1(T^R(\rA))$
such that for each first order formula $\varphi(x,\overline{b})$ with parameters
in $\CM$,
$$(\mu\llbracket\varphi(x,\overline{b})\rrbracket)^{\rq'}=
\begin{cases}
1 & \mbox{ if }  \varphi(x,\overline{b})\in q\\
0 & \mbox{otherwise}
\end{cases}
$$
Then $q\mapsto \rq'$
is a mapping from $S_1(T(\CM))$ into $S_1(T^R(\rA))$ such that $p\ne q$ implies $d(\rp',\rq')=1$.
 Since $S_1(T^R(\rA))$ is separable,
it follows that  $S_1(T(\CM))$ is countable, so $T$ is $\omega$-stable.

Now assume that $T$ is $\omega$-stable.  Then
$T$ has a countable saturated model $\CM_1$,
and $\CM_1$ has a  countable elementary extension  $\CM_2$ which realizes every
type in $S_1(T(\CM_1))$.  Let
$([0,1]\times[0,1],\CL\otimes\CL,\lambda\otimes\lambda)$ be
the natural Borel measure on the unit square.
Let $\CB_1$ be the subalgebra of $\CL\otimes\CL$ generated by $\{\sB\times [0,1]:\sB\in\CL\}$,
and let $\CB_2=\CL\otimes\CL$.
Let $\CK_1$ and $\CK_2$ be the sets of $\CB_1$-countable random elements of $\CM_1$ and
$\CB_2$-countable random elements of $\CM_2$ respectively.
Then $(\CK_1,\CB_1)$ and $(\CK_2,\CB_2)$ are full randomizations of $\CM_1$
and $\CM_2$ respectively.  By Theorems \ref{qe} and \ref{complete},  $(\CK_1,\CB_1)$
is an elementary pre-substructure of $(\CK_2,\CB_2)$. By Lemma \ref{l-omega-sat},
$(\CK_1,\CB_1)$ is $\omega$-saturated.   By Lemma \ref{simple},  $\CK_2$ has a countable
dense subset.  It therefore suffices to show that every type in $S_1(T^R(\CK_1))$
is realized in $(\CK_2,\CB_2)$.

Let $\rp$ be a type in $S_1(T^R(\CK_1))$.
Let $\{q_0,q_1,\ldots\}$ be an enumeration of the countable set
of types $S_1(T(\CM_1))$, and let $a_n$ realize $q_n$ in $\CM_2$.
It follows from the axioms of $T^R$ that for each $n$, there is a unique
($\sigma$-additive) measure $\nu_n$ on $\CL$ defined by
$$ \nu_n(\sB)=\inf\{(\mu[\llbracket\varphi(x,\overline{\rg})\rrbracket\sqcap \sB])^p:\overline{\rg}
\mbox{ deterministic}, \varphi(x,\overline{\rg})\in q_n\}.$$
Then $\nu_n $ is absolutely continuous with respect to $\lambda$, and
$\nu(\sB)=\sum_n\nu_n(\sB)$ is a probability measure on $\CL$.
By the Radon-Nikodym theorem, there is an $\CL$-measurable
function $f_n:[0,1]\to [0,1]$ such that $\nu_n(\sB)=\int_\sB f_n d\lambda$.
Define $\rf\in\CK_2$ as follows.  For $(r,s)\in[0,1]\times[0,1)$, let
$\rf(r,s)=a_n$ if and only if
$$\sum_{k<n} f_k(r)\leq s < \sum_{k\leq n} f_k(r).$$
For $r\in[0,1]$ let $\rf(r,1)$ be some particular element of $\CM$,
say $\rf(r,1)=a_0$.

Consider a $k$-tuple $\overline{\rg}$ in $\CK_1$.  Let $\CM_1^k=\{\overline{b}_m:m\in\omega\}$.
  For each $r\in[0,1]$ we have $\overline{\rg}(r)\in\{\overline{b}_m:m<\omega\}$.
Let $\sB_m=\{r\in[0,1]:\overline{\rg}(r)=\overline{b}_m\}$.
Then $\sB_m\in\CL$.  For each $m$ and $n$ we have
$$\nu_n(\sB_m)=\int_{\sB_m}\, f_n\,d\lambda=
(\lambda\otimes\lambda)\{(r,s):\rf(r,s)=a_n\wedge r\in \sB_m\}.$$
By the definition of $\nu_n$, for each first order formula
$\varphi(x,\overline{b}_m)\in q_n$,
$$\nu_n(\sB_m)\leq(\mu[\llbracket\varphi(x,\overline \rg)\rrbracket\sqcap \sB_m])^p.$$
It follows that $\rf$  realizes $\rp$ in $(\CK_2,\CB_2)$.
\end{proof}

Isaac Goldbring has noted that in the above proof,
$$ \nu_n(\sB)=(\lambda\otimes\lambda)\{(r,s):r\in \sB\wedge \rf(r,s)\models q_n\},$$
and hence
$$ \nu(\sB)=\sum_n\nu_n(\sB)=(\lambda\otimes\lambda)(\sB\times [0,1])=\lambda(\sB).$$
Thus $\nu$ is the usual measure $\lambda$ on $\CL$ and does not depend on $\rp$.

\begin{rmk}  In \cite{BenYaacov:SchroedingersCat} it is shown that the theory of atomless measure algebras
in continuous logic is $\omega$-stable.
In view of Remark \ref{measure-algebra},
in the case that $\CM$ is the trivial two-element structure, Theorem \ref{omega-stable} gives
another proof of that fact.  So Theorem \ref{omega-stable} can be viewed as a generalization
of the result that the theory of atomless measures is $\omega$-stable.
\end{rmk}

\subsection{Incomplete Theories}

Let $S_0(T)$ be the space of complete extensions of $T$.
In this subsection we show that if $S_0(T)$ is countable, then
$T$ is $\omega$-stable if and only if $T^R$ is $\omega$-stable.

We first take a brief detour to
state a generalization of the Second Representation Theorem \ref{representation} which
shows that every model of $T^R$ can be regarded as a
continuous structure whose elements are random variables taking values in
\textit{random} models of $T$.  A \textbf{full randomization} $(\CK,\CB)$ of
an indexed family $\langle \CM(w): w\in\Omega\rangle$
of models of $T$ is defined in the same way as a full randomization of $\CM$
except that $\CK\subseteq\Pi_{w\in\Omega}\CM(w)$.

\begin{thm} \label{T-representation}
(\cite{Keisler:Randomizing}, Proposition 5.6 and Theorem 5.7, restated)  A pre-structure
$(\CK',\CB')$ is a pre-model of $T^R$ with perfect witnesses if and only if
it can be represented by some full randomization $(\CK,\CB)$ of
some indexed family $\langle \CM(w): w\in\Omega\rangle$
of models of $T$.
\end{thm}

We now introduce countable convex combinations of pre-complete models of $T^R$.
This construction will be used in proving that the randomization operation
preserves $\omega$-stability.
Let $S$ be a countable set, and let $(S,\CB_0,\mu_0)$ be a probability space
where $\CB_0$ is the power set of $S$.  For each $w\in S$ let $(\CK(w),\CB(w))$
be a pre-complete model of $T^R$.  We then define
$\int_S (\CK(w),\CB(w))\, d\mu_0(w)$ to be the pre-structure $(\CK,\CB)$
such that $\CK=\Pi_w \CK(w)$, $\CB$ is the set of events with parameters in $\CK$,
and for each tuple $\overline{\rf}$ in $\CK$ and first order formula $\varphi(\overline{x})$,
$$ \mu\llbracket\varphi(\overline{\rf})\rrbracket =
\int_S \mu(w)\llbracket\varphi(\overline{\rf}(w))\rrbracket\,d\mu_0(w).$$

The proof of the following lemma is routine.

\begin{lem} \label{convex}
Let $S, \mu_0$, and $(\CK(w),\CB(w))$ be as above, and let
$$(\CK,\CB)=\int_S (\CK(w),\CB(w))\, d\mu_0(w).$$

(i)  $(\CK,\CB)$ is a pre-complete model of $T^R$.

(ii)  If each $(\CK(w),\CB(w))$ is separable, then $(\CK,\CB)$ is separable.

(iii) If each $(\CK(w),\CB(w))$ is $\omega$-saturated, then $(\CK,\CB)$ is $\omega$-saturated.

(iv)  If $S_1(T^R(\CK(w)))$ is separable for each $w\in S$, then
$S_1(T^R(\CK))$ is separable.
\end{lem}

\begin{thm}  \label{T-omega-stable}
Suppose that $S_0(T)$ is countable.  Then
$T$ is $\omega$-stable if and only if $T^R$ is $\omega$-stable.
\end{thm}

\begin{proof}  Assume $T^R$ is $\omega$-stable. Let $\CM$ be an arbitrary countable
model of $T$, let $\CK=\CM^{[0,1]}$, and let $\rA$ be the set of deterministic
elements of $\CK$.  Then $\rA$ is countable, so $S_1(T^R(\rA))$ is separable.
  By the argument used in the proof of
Theorem \ref{omega-stable}, it follows that $S_1(T(\CM))$ is countable, so $T$ is $\omega$-stable.

Suppose $T$ is $\omega$-stable.  Let $S=S_0(T)$.
Then each $w\in S$ is a complete $\omega$-stable theory which
has a countable $\omega$-saturated model $\CM_1(w)$.
By Theorem \ref{complete}, each $w^R$ is a complete extension of $T^R$.
By Lemma \ref{l-omega-sat} and Corollaries \ref{c-precomplete} and \ref{c-separable},
$w^R$ has a pre-complete separable  $\omega$-saturated model $(\CK_1(w),\CB_1(w))$.
By Theorem \ref{omega-stable}, each $w^R$ is an $\omega$-stable theory.
Therefore the type space $S_1(w^R(\CK_1(w)))$ is separable.

Let $(\CK,\CB)$ be a separable model of $T^R$.
Let $\mu_0$ be the unique measure in $\FR(S)$ such that for each
sentence $\varphi$ of $L$ and $r\in [0,1]$, if
$$(\CK,\CB)\models \mu\llbracket\varphi\rrbracket \ge r$$
then $\mu_0(\{w\in S:w\models\varphi\})\ge r$.
The pre-models $(\CK,\CB)$ and $(\CK_1,\CB_1)$ assign the same measure to
each sentence $\varphi$ of $L$.  Therefore by quantifier elimination,
$(\CK_1,\CB_1)$ is elementarily equivalent to $(\CK,\CB)$.

By Lemma \ref{convex}, $(\CK_1,\CB_1)=\int_S (\CK_1(w),\CB_1(w))\, d\mu_0(w)$
is an $\omega$-saturated separable pre-complete model of $T$, and the type space
$S_1(T^R(\CK_1))$ is separable. Since $(\CK,\CB)$
is separable, $(\CK,\CB)$ is elementarily pre-embeddable in $(\CK_1,\CB_1)$.
Hence the type space $S_1(T^R(\CK))$ is separable.  This shows that $T^R$
is $\omega$-stable.
\end{proof}

\section{Independence and Stability}

\subsection{Conditional expectations and fiber products of measures}

We have seen earlier that types in the sense of $T^R$ are measures on
the type spaces of $T$.
Here we will give a few results in an abstract measure-theoretic
setting that will be useful later on.

\begin{fct}
  Let $(X,\Sigma_X)$ and $(Y,\Sigma_Y)$ be two measurable spaces,
  $\pi\colon X \to Y$ a measurable
  function and let $\mu$ a probability measure on $X$.
  Let $\hat \pi(\mu) = \mu \circ \pi^{-1}$ denote the
  image measure on $Y$.

  Then for every $f \in L^p(X,\Sigma_X,\mu)$ there is a
  unique (up to equality a.e.)\ function
  $g \in L^p(Y,\Sigma_Y,\hat \pi(\mu))$
  such that:
  \begin{align*}
    & \int_S g\,d\hat \pi(\mu) = \int_{\pi^{-1}(S)} f\, d\mu
    && \forall S \in \Sigma_Y
  \end{align*}
  This unique function will be denoted
  $g = \bE^\mu[f|\pi]$.
  The mapping $\bE^\mu[\cdot|\pi]$ has all the usual properties of conditional
  expectation: it is additive,
  linear in the sense that
  $\bE^\mu[(h\circ\pi)\cdot f|\pi] = h\cdot\bE^\mu[f|\pi]$ for
  $h \in L^q(Y,\Sigma_Y,\hat \pi(\mu))$,
  satisfies Jensen's inequality and so on.
\end{fct}
\begin{proof}
  Identical to the classical proof of the existence of conditional
  expectation.
  Indeed, the classical case is merely the one where the underlying
  sets $X$ and $Y$ are equal and $\pi$ is the identity.
\end{proof}
When the mapping $\pi \colon X \to Y$ is clear from the context we may
write $\bE^\mu[\cdot|Y]$  instead of $\bE^\mu[\cdot|\pi]$.

Let us add assumptions.
Now $X$ and $Y$ are compact Hausdorff topological spaces,
equipped with their respective $\sigma$-algebras of Borel sets,
and $\pi\colon X \to Y$ is continuous (so in particular Borel measurable).
Let $\FR(X)$ denote the set of regular Borel probability measures on
$X$, and similarly $\FR(Y)$ and so on.
Since $\pi$ is continuous the image measure of a regular Borel measure
is regular as well, so we have $\hat \pi \colon \FR(X) \to \FR(Y)$.

\begin{ntn}
  For a continuous function $\varphi\colon X \to \bR$ and
  $\mu \in \FR(X)$ let us define
  \begin{gather*}
    \langle\varphi,\mu\rangle = \int \varphi\,d\mu.
  \end{gather*}
\end{ntn}

We equip $\FR(X)$ with the weak-$*$ topology, namely with the minimal
topology under which the mapping
$\langle\varphi,\cdot\rangle\colon \mu \mapsto \int \varphi\,d\mu$
is continuous for every continuous
$\varphi\colon X \to [0,1]$
(equivalently, for every continuous $\varphi\colon X \to \bR$).
It is a classical fact that this is a compact Hausdorff topology.

\begin{fct}[Riesz Representation Theorem]
  The mapping
  $\mu \mapsto \langle\cdot,\mu\rangle$
  defines a bijection between $\FR(X)$ and the positive
  linear functionals on $C(X,\bR)$ satisfying $\lambda(1) = 1$.
\end{fct}
\begin{proof}
  Rudin, \cite[Theorem~2.14]{Rudin:RealAndComplexAnalysis}.
\end{proof}

Equipping the space of positive functionals with the topology of
point-wise convergence this bijection
is a homeomorphism, whence the compactness
of $\FR(X)$ follows easily.

Let us now consider a ``conditional'' variant of the functional
$\langle\varphi,\cdot\rangle$.
For $\nu \in \FR(Y)$ let $\FR_\nu(X)$ denote the fiber above $\nu$.
For a continuous $\varphi\colon X \to [0,1]$ and
$\mu \in \FR_\nu(X)$ let
$\hat \varphi_\nu(\mu) = \bE^\mu[\varphi|Y]$.

\begin{lem}
  \label{lem:ContinuousConditionalExpectation}
  Let $\varphi\colon X \to [0,1]$ be a continuous function and
  $\nu \in \FR(Y)$.
  Then
  $\hat \varphi_\nu\colon \FR_\nu(X) \to L^1(Y,\nu)$ is
  continuous where $L^1(Y,\nu)$ is equipped with the weak topology.
  In other words, for every $\psi \in L^\infty(Y,\nu)$
  the mapping
  $\mu \mapsto \int_Y \bE^\mu[\varphi|Y]\psi\, d\nu$
  is continuous.
\end{lem}
\begin{proof}
  Let $\psi \in L^\infty(Y,[0,1])$, and let $\varepsilon > 0$ be
  given.
  Since $\nu$ is regular there is a continuous function
  $\tilde \psi\colon Y \to \bC$ which is close in $\nu$ to $\psi$,
  i.e., such that
  $\nu\bigl\{ y\colon |\psi(y)-\tilde \psi(y)| > \varepsilon \bigr\}
  < \varepsilon$.
  Then for every measurable $\chi\colon Y \to [0,1]$:
  \begin{gather*}
    \left| \int_Y \chi\psi\, d\nu
      - \int_Y \chi\tilde \psi\, d\nu \right|
    < 2\varepsilon.
  \end{gather*}
  The product $\tilde \varphi = \varphi\cdot\tilde\psi\circ\pi$
  is continuous on $X$ so $\mu$ admits a neighborhood
  \begin{gather*}
    U = \bigl\{ \mu' \in \FR_\nu(X) \colon
    |\langle \tilde \varphi,\mu \rangle
    - \langle \tilde \varphi,\mu' \rangle| < \varepsilon
    \big\}.
  \end{gather*}
  We conclude observing that for every $\mu' \in U$
  \begin{align*}
    \left| \int_Y \bE^\mu[\varphi|Y]\psi\, d\nu
      - \int_Y \bE^{\mu'}[\varphi|Y]\psi\, d\nu \right| &
    < \left| \int_Y \bE^\mu[\varphi|Y]\tilde \psi\, d\nu
      - \int_Y \bE^{\mu'}[\varphi|Y]\tilde \psi\, d\nu \right|
    + 4\varepsilon
    \\ &
    = \left| \langle \tilde \varphi,\mu \rangle
      - \langle \tilde \varphi,\mu' \rangle \right|
    + 4\varepsilon
    < 5\varepsilon.
    \qedhere
  \end{align*}
\end{proof}

Sometimes we will find ourselves in a situation where we have some
constraints on a measure and we wish to decide whether there is a
measure $\mu \in \FR(X)$ satisfying these constraints.
This will usually take the following form:

\begin{prp}
  \label{prp:MeasureIneqExtension}
  Let $X$ be a compact Hausdorff space,
  $A \subseteq C(X,\bR)$ any subset and $\lambda_0 \colon A \to \bR$ any mapping.
  Then the following are equivalent:
  \begin{enumerate}
  \item There exists a measure $\mu \in \FR(X)$ satisfying
    $\langle\varphi,\mu\rangle \leq \lambda_0(\varphi)$ for every $\varphi \in A$.
  \item Whenever $\{\varphi_i\}_{i<\ell} \subseteq A$ are such that
    $\sum_{i<\ell} \varphi_i \geq n$ we also have $\sum_{i<\ell} \lambda_0(\varphi_i) \geq n$.
  \end{enumerate}
\end{prp}
\begin{proof}
  One direction is clear so we prove the other.
  Let $S$ denote the set of all partial functions $\lambda$ which satisfy
  the condition in the second item, noticing that it is equivalent to:
  \begin{gather*}
    \sum_{i<\ell} \alpha_i\lambda(\varphi_i) \geq \inf_{x\in X} \sum_{i<\ell}
     \alpha_i\varphi_i(x) \qquad \text{for all }
    \{(\alpha_i,\varphi_i)\}_{i<\ell} \subseteq \bR^+ \times dom(\lambda).
  \end{gather*}
  For $\lambda \in S$, $\psi \in C(X,\bR)$ and $\beta \in \bR$ define $\lambda_{\psi,\beta}$ be the
  partial functional which coincides with $\lambda$ except at $\psi$ and
  satisfying $\lambda_{\psi,\beta}(\psi) = \beta$
  (so $dom(\lambda_{\psi,\beta}) = dom(\lambda) \cup \{\psi\}$).

  Given $\lambda \in S$ and $\psi \in C(X,\bR)$ one can always find $\beta$
  such that $\lambda_{\psi,\beta} \in S$ as well.
  For example, $\beta = \sup \psi$ will
  always do.
  The least such $\beta$ is always given by:
  \begin{gather*}
    \tilde \lambda(\psi) = \sup \left\{
      \inf \left( \psi + \sum_{i<\ell} \alpha_i\varphi_i \right) - \sum_{i<\ell}
      \alpha_i\lambda(\varphi_i)\colon
      \{(\alpha_i,\varphi_i)\}_{i<\ell} \subseteq \bR^+ \times dom(\lambda)
    \right\}.
  \end{gather*}
  In particular if $\psi \in dom(\lambda)$ then $\tilde \lambda(\psi) \leq \lambda(\psi)$.
  It is also immediate to check that:
  \begin{align*}
    & \tilde \lambda(\alpha) = \alpha & \alpha \in \bR \\
    & \tilde \lambda(\alpha\psi) = \alpha\tilde \lambda(\psi) & \alpha \in \bR^+ \\
    & \tilde \lambda(\psi + \psi') \geq \tilde \lambda(\psi) + \tilde \lambda(\psi') \\
    & \tilde \lambda(\psi + \psi') \leq \lambda(\psi) + \lambda(\psi') & \psi,\psi' \in dom(\lambda)
  \end{align*}
  For $\lambda,\lambda' \in S$ say that $\lambda \preceq \lambda'$ if $dom(\lambda) \subseteq dom(\lambda')$ and
  $\lambda(\varphi) \geq \lambda'(\varphi)$ for every $\varphi \in dom(\lambda)$.
  By Zorn's Lemma $(S,\preceq)$ admits a maximal element $\lambda \succeq \lambda_0$.
  By its maximality $\lambda$ is total and satisfies
  $\tilde \lambda(\varphi) = \lambda(\varphi)$.

  It follows that $\lambda(\varphi+\psi) = \lambda(\varphi) + \lambda(\psi)$ and that
  $\lambda(\alpha\varphi) = \alpha\lambda(\varphi)$, first for $\alpha \geq 0$ and then for every
  $\alpha \in \bR$.
  Therefore $\lambda$ is a linear functional, and it is positive as it
  belongs to $S$.
  Since in addition $\lambda(1) = 1$,
  it is necessarily of the form
  $\lambda =
  \langle\cdot,\mu\rangle \colon
  \varphi \mapsto \langle\varphi,\mu\rangle$
  for some measure
  $\mu \in \FR(X)$.
  Thus $\langle\cdot,\mu\rangle \succeq \lambda_0$, so $\mu$ is as desired.
\end{proof}

One can also prove a conditional variant of
\fref{prp:MeasureIneqExtension} but such a result will not be required
here.
Instead, let us re-phrase \fref{prp:MeasureIneqExtension} a little:

\begin{prp}
  \label{prp:MeasureEqExtension}
  Let $X$ be a compact Hausdorff space,
  $A \subseteq C(X,\bR)$ any subset and $\lambda_0 \colon A \to \bR$ any mapping sending
  $1 \mapsto 1$.
  Then the following are equivalent:
  \begin{enumerate}
  \item There exists a measure $\mu \in \FR(X)$ satisfying
    $\langle\mu,\varphi\rangle = \lambda_0(\varphi)$ for every $\varphi \in A$.
  \item Whenever $\{(\varphi_i,m_i)\}_{i<\ell} \subseteq A \times Z$ are such that
    $\sum_{i<\ell} m_i\varphi_i \geq 0$ we also have $\sum_{i<\ell} m_i\lambda_0(\varphi_i) \geq 0$.
  \end{enumerate}
\end{prp}
\begin{proof}
  One direction is clear.
  For the other observe that if $\varphi,-\varphi \in A$ then
  necessarily $\lambda_0(-\varphi) +\lambda_0(\varphi) = 0$.
  We may therefore define $\lambda_1$ whose domain
  is $A \cup -A$ by $\lambda_1(\varphi) = \lambda_0(\varphi)$,
  $\lambda_1(-\varphi) = -\lambda_0(\varphi)$.
  Then $\lambda_1$ satisfies the conditions of the previous Proposition, and
  the corresponding measure $\mu$ is as desired.
\end{proof}

Let us now consider fiber products of measures.
Again we consider general measurable spaces (i.e., not necessarily
topological) $X = (X,\Sigma_X)$, $Y = (Y,\Sigma_Y)$, $Z = (Z,\Sigma_Z)$,
equipped with measurable mappings $\pi_X \colon X \to Z$, $\pi_Y\colon Y \to Z$.
Let $X \times_Z Y$ denote the set theoretic fiber product:
\begin{gather*}
  X \times_Z Y = \{(x,y) \in X \times Y\colon \pi_X(x) = \pi_Y(y)\}.
\end{gather*}
Let $\Sigma_0 = \{A \times_Z B\colon A \in \Sigma_X, B \in \Sigma_Y\}$ and let $\Sigma_X \otimes_Z \Sigma_Y$
denote the generated
$\sigma$-algebra on $X \times_Z Y$.
The set $X \times_Z Y$ is thereby rendered a measurable space
with the canonical mappings to $X$ and $Y$ measurable (it is the
measurable space fiber product).

Let $\mu$ and $\nu$ be probability measures on $X$ and $Y$, respectively,
with the same image measure on $Z$: $\hat \pi_X(\mu) = \hat \pi_Y(\nu)$.
For  $A \in \Sigma_X$ and $B \in \Sigma_Y$ define:
\begin{align*}
  (\mu \otimes_Z \nu)_0(A \times_Z B)
  & = \int_Z \bP^\mu[A|Z]\bP^\nu[B|Z] \, d\hat \pi_X(\mu) \\
  & = \int_A \bP^\nu[B|Z] \circ \pi_X \, d\mu \\
  & = \int_B \bP^\mu[A|Z] \circ \pi_Y \, d\nu.
\end{align*}
It is not difficult to check that $\Sigma_0$ is a semi-ring and that
$(\mu \otimes_Z \nu)_0$ defined in this manner is a $\sigma$-additive probability
measure on $\Sigma_0$.
By Carathéodory's theorem it extends to a $\sigma$-additive probability
measure on a $\sigma$-algebra containing $\Sigma_0$, and in particular to
$\Sigma_X \otimes_Z \Sigma_Y$.
We will denote this fiber product measure by $\mu \otimes_Z \nu$.
Its image measures on $X$ and $Y$ are $\mu$ and $\nu$, respectively.
If we let $\FM(X)$ denote the collection of probability measures on
$X$ and so on we have obtained a mapping:
\begin{gather*}
  \begin{array}{ccc}
    \FM(X) \times_{\FM(Z)} \FM(Y) & \to & \FM\left( X \times_Z Y \right) \\
    &\\
    (\mu,\nu) & \mapsto & \mu \otimes_Z \nu
  \end{array}
\end{gather*}

Let us switch back to the topological setting, where all spaces are
compact Hausdorff spaces equipped with the $\sigma$-algebras of Borel sets
and the mappings are continuous.
In particular we have $\FR(X) \times_{\FR(Z)} \FR(Y) \subseteq \FM(X) \times_{\FM(Z)} \FM(Y)$.
(Going through the construction one should be able to verify that if
$(\mu,\nu) \in \FR(X) \times_{\FR(Z)} \FR(Y)$ then
$\mu \otimes_Z \nu$ is a regular measure on $X \times_Z Y$ as well, but we will not
need this observation.)
Given a Borel function $\rho\colon X \times_Z Y \to [0,1]$ we may define:
\begin{gather*}
  \begin{array}{rccc}
    \hat \rho\colon & \FR(X) \times_{\FR(Z)} \FR(Y) & \to & [0,1] \\
    &\\
    & (\mu,\nu) & \mapsto & \bE^{\mu \otimes_Z \nu}[\rho]
  \end{array}
\end{gather*}

\begin{lem}
  \label{lem:ContinuityInOneVariable}
  Let $X$, $Y$ and $Z$ be compact Hausdorff spaces and
  $X \times_Z Y$ defined as above.
  Let $\rho\colon X \times_Z Y \to [0,1]$ be a Borel function.

  Let us fix $\mu \in \FR(X)$, letting $\eta = \hat\pi_X(\mu)$,
  and assume we can find a countable family
  of Borel subsets $X_i \subseteq X$ such that
  \begin{enumerate}
  \item $\mu\left( \bigcup X_i \right) = 1$.
  \item For each $i$
    there is a continuous function $\rho_i(y)\colon Y \to [0,1]$ such
    that $\rho(x,y) = \rho_i(y)$ for every
    $(x,y) \in X_i \times_Z Y$.
  \end{enumerate}
  Then $\hat \rho(\mu,\cdot) \colon \FR_\eta(Y) \to [0,1]$ is continuous.
\end{lem}
\begin{proof}
  We may assume that all the $X_i$ are disjoint.
  Then:
  \begin{align*}
    \hat \rho(\mu,\nu)
    & =
    \int \rho \, d(\mu \otimes_Z \nu)
    \\ & =
    \sum_{i\in I}     \int_{X_i \times_Z Y} \rho \, d(\mu \otimes_Z \nu)
    \\ & =
    \sum_{i\in I} \int \rho_i \cdot\bP^{\mu \otimes_Z \nu}[X_i \times_Z Y | Y] \, d\nu
    \\ & =
    \sum_{i\in I} \int \bE^\nu[\rho_i|Z] \cdot \bP^{\mu \otimes_Z \nu}[X_i \times_Z Y | Z] \, d\eta
    \\ & =
    \sum_{i\in I} \int \bE^\nu[\rho_i|Z] \cdot \bP^\mu[X_i| Z] \, d\eta
  \end{align*}
  Fixing $i \in I$, the mapping
  $\FR_\eta(Y) \to L^1(Z,\Sigma_Z,\eta)$ defined by
  $\nu \mapsto \bE^\nu[\rho_i|Z]$ is continuous in the weak topology
  by \fref{lem:ContinuousConditionalExpectation}.
  It follows that
  $\nu \mapsto \int \bE^\nu[\rho_i|Z] \cdot \bP^\mu[X_i| Z] \, d\eta$ is continuous.
  Finally, the series above converges absolutely and uniformly so,
  whence the continuity of $\nu \mapsto \rho(\mu,\nu)$.
\end{proof}

\begin{prp}
  \label{prp:ContinuityInOneVariable}
  Let $X$, $Y$ and $Z$ be compact Hausdorff spaces and
  $X \times_Z Y$ defined as above.
  Let $\rho\colon X \times_Z Y \to [0,1]$ be any function.
  Assume that:
  \begin{enumerate}
  \item The set $X$ is covered by a countable family of Borel sets
    $X = \bigcup_{i\in \bN} X_i$.
  \item Each $X_i$ is covered by a (possibly uncountable) family of
    relatively open subsets $X_i = \bigcup_{j\in J_i} G_{i,j}$.
  \item For each pair $(i,j)$ (where $i \in \bN$ and $j \in J_i$)
    there is a continuous function
    $\rho_{i,j}\colon Y \to [0,1]$ such that
    $\rho(x,y) = \rho_{i,j}(y)$ for every
    $(x,y) \in G_{i,j} \times_Z Y$.
  \end{enumerate}
  Then $\rho$ is Borel and
  for every $\eta \in \FR(Z)$ and $\mu \in \FR_\eta(X)$,
  the mapping
  $\hat \rho(\mu,\cdot) \colon \FR_\eta(Y) \to [0,1]$ is continuous.
\end{prp}
\begin{proof}
  We may assume that the $X_i$ are all disjoint.

  Restricted to $G_{i,j} \times_Z Y$, $\rho$ is continuous.
  Since each $G_{i,j}$ is relatively open in $X_i$,
  the set $G_{i,j} \times_Z Y$ is relatively open in
  $X_i \times_Z Y$.
  Thus $\rho$ restricted to $X_i \times_Z Y$ is continuous for each $i$.
  Moreover, each $X_i \times_Z Y$ is Borel in $X \times_Z Y$.
  If follows that $\rho$ is Borel.

  Now fix $\eta \in \FR(Z)$, $\mu \in \FR_\eta(X)$.
  Fix again $i \in \bN$.
  Since $\mu$ is regular, $\mu(X_i)$ is the supremum of
  $\mu(K)$ where $K \subseteq X_i$ is compact.
  Each such $K$ can be covered by finitely many of the $G_{i,j}$.
  Thus there is a countable family $J_i^0 \subseteq J_i$ such that
  $\mu(X_i) = \mu\left( \bigcup_{j\in J_i^0} G_{i,j} \right)$.
  It follows that
  \begin{gather*}
    \mu\left( \bigcup_{i\in \bN, j \in J^0_i} G_{i,j} \right) = 1.
  \end{gather*}
  Moreover each $G_{i,j}$ is relatively open in a Borel set and
  therefore a Borel set itself.
  The conditions of the Lemma are therefore fulfilled.
\end{proof}

\subsection{Some reminders regarding stable formulas}

Let $T$ denote a first order theory.
Let $\varphi(x,y)$ be a stable formula in $T$.
We may also assume that $T$ eliminates imaginaries.

Since $\varphi$ is stable, every $\varphi$-type over a model
$\hat p \in S_\varphi(\CM)$ is definable by a unique formula
$d_{\hat p}\varphi(y)$:
\begin{gather*}
  \hat p(x) = \{\varphi(x,b)\leftrightarrow d_{\hat p}\varphi(b)\}_{b\in\CM}.
\end{gather*}

Let $A \subseteq \CM$ denote an algebraically closed set, and we may assume
that $\CM$ is sufficiently saturated and homogeneous over $A$.
Then for every complete type $p \in S_x(A)$ there exists
a unique $\varphi$-type $\hat p \in S_\varphi(\CM)$ which is definable over
$A$ and compatible with $p$.
The (unique) definition of $\hat p$ is also referred to as the
$\varphi$-definition of $p$, denoted $d_p\varphi(y)$.

Now let $A \subseteq \CM$ be any set, not necessarily algebraically closed,
and $p \in S_x(A)$.
Let $\hat P \subseteq S_\varphi(\CM)$ denote the set of global $\varphi$-types
$\hat p$ which are definable over $acl(A)$ and compatible with $p$.
Then $\hat P$ is non-empty, finite, and $Aut(\CM/A)$ acts
transitively on $\hat P$.
We define a $[0,1]$-valued continuous predicate $\rho(p,y)$ by:
\begin{gather*}
  \rho(p,b)
  = \frac{|\{\hat p \in \hat P\colon \hat p \vdash \varphi(x,b)\}|}{|\hat P|}.
\end{gather*}
We consider $\rho(p,b)$ to be the \emph{probability} that
a randomly chosen non-forking $\varphi$-extension of $p$ should satisfy
$\varphi(x,b)$.

Shelah's notions of local rank
$R(\cdot,\varphi)$ and multiplicity $M(\cdot,\varphi)$, both with values in
$\bN \cup \{\infty\}$, provide a more useful way to calculate $\rho(p,b)$.
For our purpose it will be more convenient to define the multiplicity
of a partial type $\pi(x)$ at each rank $n$,
denoted $M(\cdot,\varphi,n)$.
\begin{enumerate}
\item If $\pi$ is consistent then $R(\pi,\varphi) \geq 0$.
\item Having defined when $R(\cdot,\varphi) \geq n$ we define
  $M(\cdot,\varphi,n)$.
  We say that $M(\cdot,\varphi,n) \geq m$ if there are types
  $\pi(x) \subseteq \pi_i(x)$ for $i < m$ such that for every $i < j < m$ there is
  a tuple $b_{ij}$ for which
  $\pi_i(x) \cup \pi_j(x') \vdash  \varphi(x,b_{ij})\leftrightarrow\neg\varphi(x',b_{ij})$,
  and in addition $R(\pi_i,\varphi) \geq n$ for all $i < m$.
\item If $M(\pi,\varphi,n) = \infty$ then $R(\pi,\varphi) \geq n+1$.
\end{enumerate}
The formula $\varphi$ is stable if and only if $R(\pi,\varphi)$ is always finite.
If $R(\pi,\varphi) = n$ then $M(\pi,\varphi) = M(\pi,\varphi,n)$ is the $\varphi$-multiplicity of
$\pi$.
Notice that $n < R(\pi,\varphi) \Longleftrightarrow M(\pi,\varphi,n) = \infty$ and
$n > R(\pi,\varphi) \Longleftrightarrow M(\pi,\varphi,n) = 0$.

If $[\pi] \subseteq S_\varphi(\CM)$ denoted the (closed) set of global $\varphi$-types
consistent with $\pi$ then it is not difficult to check that
$R(\pi,\varphi)$ is precisely the Cantor-Bendixson rank of $[\pi]$ in
$S_\varphi(\CM)$.
Similarly, $M(\pi,\varphi)$ is the Cantor-Bendixson multiplicity, namely the
number of types in $[\pi]$ of rank $R(\pi,\varphi)$.
In case $p \in S_x(A)$ is a complete type then
$\hat P$ is precisely is set of $\varphi$-types of maximal
Cantor-Bendixson rank $n = R(p,\varphi)$ in $[p]$,
so $M(p,\varphi) = |\hat P|$.
The number of $\hat p \in \hat P$ which contain some instance
$\varphi(x,b)$ is then precisely the number of types of rank $n$ in
$[p(x) \cup \{\varphi(x,b)\}]$, whereby:
\begin{gather*}
  \rho(p,b) = \frac{M\big( p\cup\{\varphi(x,b)\},\varphi,n \big)}{M(p,\varphi,n)}.
\end{gather*}
Finally, let $\xi(x,\bar a) \in p$ be such that
$M\big( \xi(x,\bar a),\varphi,n \big) = M(p,\varphi,n)$,
namely, such that $[\xi(x,\bar a)]$ and $[p]$ contain precisely
the same types of rank $n$.
Then we have:
\begin{gather*}
  \rho(p,b) = \frac{M\big( \xi(x,\bar a) \wedge\varphi(x,b),\varphi,n \big)}
  {M\big( \xi(x,\bar a),\varphi,n \big)}.
\end{gather*}

The value $\rho(p,b)$ only depends on $q(y) = tp(b/A)$, so it is
legitimate to write $\rho(p,q)$.
We may re-write $p$ and $q$ as $p(x,A)$ and $q(y,A)$, where
$W$ is a tuple of variables corresponding to $A$,
$p(x,W) \in S_{x,W}(T)$ and $q(y,W) \in S_{y,W}(T)$ are complete types
and $p{\restriction}_W = q{\restriction}_W \in S_W(T)$.
In other words, $(p,q) \in S_{x,W}(T) \times_{S_W(T)} S_{y,W}(T)$.
To avoid this fairly cumbersome notation let us
write $\Sigma_{x,y,W}$ for the fiber product $S_{x,W}(T) \times_{S_W(T)} S_{y,W}(T)$.

Conversely, let $(p,q) \in \Sigma_{x,y,W}$ be any pair in the fiber product
and let $A$  realize their common restriction to $W$.
Then $\rho\big( p(x,A), q(y,A) \big)$ only depends on $(p,q)$ and not on
the choice of $A$.
In other words, $\rho\big( p(x,W), q(y,W) \big)$ makes sense for every
$(p,q) \in \Sigma_{x,y,W}$.
Henceforth we will therefore consider $\rho$ as a function
$\rho\colon \Sigma_{x,y,W} \to [0,1]$.

The next step is to show that $\rho$
satisfies the assumptions of \fref{prp:ContinuityInOneVariable}.

We define $\Xi_{n,m}$ as the set of all formulas
$\xi(x,\bar w)$ for which $M\big( \xi(x,\bar w),\varphi,n \big) \geq m$
is impossible,
namely such that for every choice of parameters
$\bar a$ in a model of $T$:
\begin{gather*}
  M\big( \xi(x,\bar a),\varphi,n \big) < m.
\end{gather*}
Clearly, if $\xi \in \Xi_{n,(m+1)}$ then
$M\big( \xi(x,\bar w),\varphi,n \big) \geq m$ is equivalent to
$M\big( \xi(x,\bar w),\varphi,n \big) = m$.

\begin{lem}
  Let $\chi(x,t)$ be any formula.
  Then the property
  $M\big( \chi(x,t),\varphi,n \big) \geq m$ is type-definable in $t$.
  In other words, there is a partial type $\pi(t)$ such that for every
  parameter $c$:
  \begin{gather*}
    M\big( \chi(x,c),\varphi,n \big) \geq m
    \quad \Longleftrightarrow \quad
    \pi(c).
  \end{gather*}
\end{lem}
\begin{proof}
  Standard.
\end{proof}

\begin{lem}
  Assume that $\xi(x,\bar w) \in \Xi_{n,(m+1)}$.
  Then there are formulas $\hat \xi_{n,m,\ell}(y,\bar w)$ for $0 \leq \ell \leq m$
  such that:
  \begin{enumerate}
  \item The formulas $\{\hat \xi_{n,m,\ell}\colon 0 \leq \ell \leq m\}$ define a partition:
    one and only one of them holds for every $b,\bar a$ in a model of
    $T$.
  \item Modulo
    $M\big( \xi(x,\bar w),\varphi,n \big) = m$,
    the formula $\hat \xi_{n,m,\ell}(y,\bar w)$ defines the property
    $M\big( \xi(x,\bar w)\wedge \varphi(x,y) ,\varphi,n \big) = \ell$.
  \end{enumerate}
\end{lem}
\begin{proof}
  For $0 \leq \ell \leq m$ let $\pi_\ell(\bar w,y)$ be the partial type expressing
  that:
  \begin{gather*}
    M\big( \xi(x,\bar w) \wedge \varphi(x,y),\varphi,n \big) \geq \ell \quad \& \quad
    M\big( \xi(x,\bar w) \wedge \neg\varphi(x,y),\varphi,n \big) \geq m-\ell.
  \end{gather*}
  If $\ell < \ell'$ then
  $\pi_\ell\wedge\pi_{\ell'}$ imply that
  $M\big( \xi(x,\bar w),\varphi,n \big) \geq \ell' + m - \ell > m$ contradicting the
  assumption on $\xi$.
  Thus $\{\pi_\ell\colon 0 \leq \ell\leq m\}$ are mutually contradictory.
  We can therefore find formulas $\hat \xi_{n,m,\ell} \in \pi_\ell$ which contradict
  one another as $\ell$ varies from $0$ to $m$.
  We may further replace $\hat \xi_{n,m,m}$ with
  $\neg\bigvee_{\ell < m} \hat \xi_{n,m,\ell}$ to achieve a partition.
  On the other hand, if $M\big( \xi(x,\bar a),\varphi,n \big) = m$ and $b$ is
  arbitrary then
  $\pi_\ell(\bar a,b)$ holds for precisely one $0 \leq \ell \leq m$, whence the
  second item.
\end{proof}

\begin{lem}
  \label{lem:DefinableRelativeMultiplicity}
  Assume that $\xi(x,\bar w) \in \Xi_{n,(m+1)}$.
  Then is a continuous predicate
  $\hat \xi_{n,m}(y,\bar w)$ taking values in $\{\frac{\ell}{m}\colon0 \leq \ell \leq m\}$
  such that modulo
  $M\big( \xi(x,\bar w),\varphi,n \big) = m$:
  \begin{gather*}
    \hat \xi_{n,m}(y,\bar w) =
    \frac{M\big( \xi(x,\bar w)\wedge \varphi(x,y) ,\varphi,n \big)}{M\big( \xi(x,\bar w) ,\varphi,n \big)}
  \end{gather*}
\end{lem}
\begin{proof}
  This is just a re-statement of the previous Lemma.
\end{proof}

We define:
\begin{gather*}
  P^0_{n,m}(x,W) = \{\neg\xi(x,\bar w)\colon \bar w \subseteq W, \xi \in \Xi_{n,m}\}.
\end{gather*}

\begin{lem}
  Let $p(x,W) = tp(a/A)$ (in a model of $T$).
  Then the following are equivalent:
  \begin{enumerate}
  \item $M\big( p(x,A),\varphi,n \big) \geq m$.
  \item $p(x,W)$ contains no member of $\Xi_{n,m}$.
  \item $p(x,W) \supseteq P^0_{n,m}(x,W)$.
  \end{enumerate}
  In other words, $P^0_{n,m}(x,W)$ expresses that
  $M\big( tp(x/W), \varphi,n \big) \geq m$.
\end{lem}
\begin{proof}
  The implications (i) $\Longrightarrow$ (ii) $\Longleftrightarrow$ (iii) are immediate.
  If $M\big( p(x,A), \varphi,n \big) < m$ then the existence of a certain
  tree whose leaves satisfy $p(x,A)$ is contradictory.
  By compactness this is due to
  some formula $\xi(x,\bar w) \in p(x,W)$ in which case
  $\xi \in \Xi_{n,m}$.
\end{proof}

For $\xi \in \Xi_{n,(m+1)}$ it is convenient to define:
\begin{gather*}
  P^\xi_{n,m}(x,W) = P^0_{n,m}(x,W) \cup \{\xi(x,\bar w)\}.
\end{gather*}
Let also:
\begin{gather*}
  P_{n,m} = \bigcup_{\xi\in\Xi_{n,(m+1)}}  [P^\xi_{n,m}]= [P^0_{n,m}] \setminus [P^0_{n,(m+1)}] \subseteq S_{x,W}(T).
\end{gather*}
This is the collection of types
$p(x,W) \in S_{x,W}(T)$ such that
$M\big( p(x,W),\varphi,n \big) = m$.
While it is in general
neither open nor closed it is locally closed and therefore
a Borel set.
Thus family
$\{P_{n,m}\colon n < \omega, 1 \leq m < \omega\}$ forms a countable partition of
$S_{x,W}(T)$ into Borel sets.
On the other hand
$P_{n,m}$ is a union of relatively clopen subsets
of the closed set $[P^0_{n,m}]$.
Moreover, let $(p,q) \in \Sigma_{x,y,W}$ and assume that $p(x,W) \models P^\xi_{n,m}(x,W)$,
$\xi(x,\bar w) \in \Xi_{n,(m+1)}$.
Let $\hat \xi_{n,m}(y,\bar w)$ be as in
\fref{lem:DefinableRelativeMultiplicity}.
Then $\rho(p,q) = \hat \xi_{n,m}(y,\bar w)^q$.

We have thus checked that
$\rho\colon \Sigma_{x,y,W} \to [0,1]$
satisfies the conditions of \fref{prp:ContinuityInOneVariable}.

\subsection{The main results}  \label{ss-main}

We continues with the same assumptions, namely that $T$ is a classical
first order theory and $\varphi(x,y)$ a stable formula.
We shall use boldface characters to denote objects of the randomized
realm.
Thus a model of $T^R$ may be denoted $\rCM$, a subset thereof $\rA$, a
type $\rp$, and so on, while objects related to the original
theory will be denoted using lightface as usual.

For our purposes here it will be more convenient to consider the
theory $T^R$ in a single-sorted language in which we have
a predicate
symbol $\bP[\psi(\bar x)]$ for every formula
$\psi(\bar x)$ in the original
language.
Each such predicate is definable in the two-sorted language given
earlier by the identity
$\bP[\psi] = \mu(\llbracket \psi \rrbracket)$.
Moreover, $T^R$ admits quantifier elimination in this language as well
(for example, since the atomic formulas separate types).

Let us translate the conclusion of
\fref{prp:ContinuityInOneVariable} applied to $\rho$
to the situation at hand.
First of all we know that
$\FR\big( S_{x,W}(T) \big) = S_{x,W}(T^R)$ and so on, so:
\begin{gather*}
  \FR\big( S_{x,W}(T) \big) \times_{\FR\big( S_W(T) \big)}
  \FR\big( S_{y,W}(T) \big)
  =
  S_{x,W}(T^R) \times_{S_W(T^R)} S_{y,W}(T^R).
\end{gather*}
Let $\rA \subseteq \rCM \models T^R$ and let $\rp(x,\rA) \in S_x(\rA)$.
Let $\rr(W) = tp(\rA)$, so $\rp(x,W)$ lies in the fiber above $\rr$.
Similarly, the fiber of $S_{y,W}(T^R)$ lying above $\rr$ can be
identified with $S_y(\rA)$.
Then \fref{prp:ContinuityInOneVariable} says that
the function
$\hat \rho(\rp(x,\rA),\cdot) \colon S_y(\rA) \to [0,1]$ is continuous.
We may therefore identify it with an $\rA$-definable predicate:
\begin{gather*}
  \hat \rho\big( \rp(x,\rA),\rb \big) := \hat \rho\big( \rp,tp(\rb,\rA) \big)
\end{gather*}

\begin{prp}
  Assume that $\rb \in \rA$.
  Then $\hat \rho\big( \rp(x,\rA),\rb \big) = \bP[ \varphi(x,\rb) ]^{\rp(x,\rA)}$.
  In other words, $\hat \rho\big( \rp(x,\rA),y \big)$ is the
  $\bP[\varphi]$-definition of $\rp(x,\rA)$.
\end{prp}
\begin{proof}
  Indeed, let $\rq(x,W) = tp(\rb,\rA)$, and say that
  $w_0 \in W$ corresponds to $\rb \in A$.
  Then the measure $\rq$ is concentrated on those types
  $q(y,W)$ which satisfy $y = w_0$, and similarly
  the measure $\rp \otimes_W \rq$ on $\Sigma_{x,y,W}$ is concentrated on those
  pairs $(p,q)$ where $q$ is such.
  For such pairs we have $\rho(p,q) = 1$ if $p \models \varphi(x,w_0)$ and zero
  otherwise, so
  \begin{align*}
    \hat \rho\big( \rp(x,\rA),\rb \big)
    & =
    \bE^{\rp \otimes_W \rq}[ \rho ]
    = \bP[ \varphi(x,w_0) ]^\rp
    = \bP[ \varphi(x,\rb) ]^{\rp(x,\rA)}.
    \qedhere
  \end{align*}
\end{proof}

We now have everything we need in order to prove preservation of
stability.

\begin{thm}
  If $\varphi(x,y)$ is a stable formula of $T$ then
  $\bP[ \varphi(x,y) ]$ is a stable formula of $T^R$.
  If $T$ is a stable theory then so is $T^R$.
\end{thm}
\begin{proof}
  We have shown that if $\varphi$ is stable then every $\bP[\varphi]$-type is
  definable, so $\bP[\varphi]$ is stable as well.
  Assume now that  $T$ is stable.
  Then every atomic formula of $T^R$ is
  stable, so every quantifier-free formula is stable, and
  by quantifier elimination every formula is stable.
  Therefore $T^R$ is stable.
\end{proof}

Given $\rp(x,\rA)$ where $\rA \subseteq \rCM$, define:
\begin{gather*}
  \hat \rp(x) = \{\bP[\varphi(x,\rb)] = \hat \rho(\rp(x,\rA),\rb)\colon \rb \in \rCM\}.
\end{gather*}

\begin{lem}
  The set of conditions $\hat \rp(x)$ is a $\bP[\varphi]$-type over $\rCM$,
  i.e., $\hat \rp(x) \in S_{\bP[\varphi]}(\rCM)$.
  Moreover, it is consistent with $\rp(x,\rA)$.
\end{lem}
\begin{proof}
  It is enough to prove the moreover part.
  For this matter, it is enough to show that for any finite tuple
  $\bar \rb = \rb_0,\ldots,\rb_{n_1}$ the following is consistent:
  \begin{gather*}
    \rp(x,\rA) \cup \{\bP[\varphi(x,\rb_i)] = \hat \rho(\rp(x,\rA),\rb_i)\colon i < n_1\}.
  \end{gather*}
  Let $\rq(\bar y,W) = tp(\bar \rb,\rA)$,
  $\rq_i = tp(\rb_i,\rA)$.
  Then we need to prove that the following is consistent with $T^R$:
  \begin{gather}
    \label{eq:StationarityTRType}
    \rp(x,W) \cup \rq(\bar y,W) \cup \{\bP[\varphi(x,y_i)] = \hat \rho(\rp,\rq_i)\colon i < n_1\}.
  \end{gather}

  Given a sequence of formulas and integer numbers
  $(\chi_k,m_k)_{k < \ell}$ let $[(\chi_k,m_k)_{k < \ell}]$ denote the formula
  stating that $\sum_{k<\ell} m_k\mathbbm{1}_{\chi_k} \geq 0$ (this is indeed
  expressible by a first order formula).
  By \fref{prp:MeasureEqExtension}, in order to show that
  \fref{eq:StationarityTRType} is consistent with $T^R$ it is enough to
  check that:
  \begin{gather*}
    T \models [(\varphi(x,y_i),f_i)_{i<n_1},(\psi_j(x,\bar w),g_j)_{j<n_2},
    (\chi_k(\bar y,\bar w),h_k)_{k<n_3}] \\
    \Longrightarrow\\
    \sum_{i<n_1} f_i\hat \rho(\rp,\rq_i)
    +    \sum_{j<n_2} g_j\bP[ \psi_j(x,\bar w) ]^\rp
    +    \sum_{k<n_3} h_k\bP[ \chi_k(\bar y,\bar w) ]^\rq
    \geq 0.
  \end{gather*}
  The sum can be rewritten as:
  \begin{gather*}
    \bE^{\rp \otimes_W \rq}\left[
      \sum_{i<n_1} f_i\rho(p,q{\restriction}_{x,y_i,W})
      +    \sum_{j<n_2} g_j\mathbbm{1}_{\psi_j}(x,\bar w)^p
      +    \sum_{k<n_3} h_k\mathbbm{1}_{\chi_k}(\bar y,\bar w)^q
    \right].
  \end{gather*}
  It will therefore be enough to show for every
  $(p,q) \in \Sigma_{x,\bar y,W}$:
  \begin{gather*}
      \sum_{i<n_1} f_i\rho(p,q{\restriction}_{x,y_i,W})
      +    \sum_{j<n_2} g_j\mathbbm{1}_{\psi_j}(x,\bar w)^p
      +    \sum_{k<n_3} h_k\mathbbm{1}_{\chi_k}(\bar y,\bar w)^q
      \geq 0
  \end{gather*}
  We may assume that $q(\bar y,W) = tp(\bar b,A)$ where
  $\bar b,A \subseteq \CM \models T$.
  Let $\hat P \subseteq S_\varphi(\CM)$ be the set of $\varphi$-types compatible with
  $p(x,A)$ and definable over $acl(A)$.
  Then the left hand side becomes:
  \begin{gather*}
    \sum_{i<n_1} f_i\rho_\varphi(p(x,A),b_i)
    +    \sum_{j<n_2} g_j\mathbbm{1}_{\psi_j}(x,\bar w)^p
    +    \sum_{k<n_3} h_k\mathbbm{1}_{\chi_k}(\bar y,\bar w)^q  \\
    = \sum_{i<n_1} f_i\frac{\{\hat p \in \hat P\colon \hat p \vdash \varphi(x,b_i)\}}{|\hat P|}
    +    \sum_{j<n_2} g_j\mathbbm{1}_{\psi_j}(x,\bar w)^p
    +    \sum_{k<n_3} h_k\mathbbm{1}_{\chi_k}(\bar y,\bar w)^q  \\
    = \frac{1}{|\hat P|}\sum_{\hat p \in \hat P}
    \left(
      \sum_{i<n_1} f_i\mathbbm{1}_\varphi(x,b_i)^{\hat p(x)}
      +    \sum_{j<n_2} g_j\mathbbm{1}_{\psi_j}(x,\bar w)^{p(x,W)}
      +    \sum_{k<n_3} h_k\mathbbm{1}_{\chi_k}(\bar y,\bar w)^{q(\bar y,W)}
    \right)
  \end{gather*}
  It will therefore be enough to show that the sum inside the
  parentheses is  positive for every $\hat p \in \hat P$.
  Since $\hat p$ is compatible with $p(x,A)$ there is a complete type
  $\tilde r(x,\CM)$ extending both.
  Let $r(x,\bar b,A)$ be its restriction to the parameters which
  interest us, where as usual $r(x,\bar y,W) \in S_{x,\bar y,W}(T)$.
  Then we have:
  \begin{gather*}
    \sum_{i<n_1} f_i\mathbbm{1}_\varphi(x,b_i)^{\hat p(x)}
    +    \sum_{j<n_2} g_j\mathbbm{1}_{\psi_j}(x,\bar w)^{p(x,W)}
    +    \sum_{k<n_3} h_k\mathbbm{1}_{\chi_k}(\bar y,\bar w)^{q(\bar y,W)} \\
    = \sum_{i<n_1} f_i\mathbbm{1}_\varphi(x,y_i)^r
    +    \sum_{j<n_2} g_j\mathbbm{1}_{\psi_j}(x,\bar w)^r
    +    \sum_{k<n_3} h_k\mathbbm{1}_{\chi_k}(\bar y,\bar w)^r.
  \end{gather*}
  Since $r$ is a type of $T$ this is indeed always positive.
\end{proof}

\begin{thm}
  Assume $\varphi$ is a stable formula of $T$, and let
  $\rA \subseteq \rCM \models T^R$.
  Then every $\bP[\varphi]$-type over $\rA$ is stationary.
  If $T$ is stable then every type over $\rA$ is stationary.
\end{thm}
\begin{proof}
  We have shown that if $\rp(x) \in S_x(\rA)$ then there exists a
  $\bP[\varphi]$-type $\hat \rp(x) \in S_{\bP[\varphi]}(\rCM)$ which is definable over
  $\rA$ (rather than merely over
  $acl(\rA)$) and compatible with $\rp$.
\end{proof}
Here $\rA$ consists of random elements in sorts of $T$.
The theory $T^R$ necessarily introduces new imaginary sorts (even if
$T$ admits elimination of imaginaries) and types over elements from
these new sorts need not be stationary.

In the course of the proof we have given an explicit characterization
of the unique non forking extension of a type in $T^R$.
Let us restate this characterization in a slightly modified fashion.
First of all we observe that the entire development above goes through
for a formula of the form $\varphi(x,y,\bar w)$ where $\bar w \subseteq W$ is a
fixed tuple of parameter variables (alternatively, we could name the
tuple $\bar w$ by new constants).
Define $\rho_\varphi\colon \Sigma_{x,y,W} \to [0,1]$ and $\hat \rho_\varphi$ accordingly.

\begin{cor}
  Let $\rc,\rb,\rA \subseteq \rCM$.
  Then $\rc \ind_\rA \rb$ if and only if
  for every $\bar \ra \subseteq \rA$ with corresponding $\bar w \subseteq W$, and for
  every formula $\bP[\varphi(x,y,\bar w)]$:
  \begin{gather*}
    \bP[\varphi(\rc,\rb,\bar \ra)] = \hat \rho_\varphi(tp(\rc,\rA),tp(\rb,\rA)).
  \end{gather*}
\end{cor}

\providecommand{\bysame}{\leavevmode\hbox to3em{\hrulefill}\thinspace}
\providecommand{\MR}{\relax\ifhmode\unskip\space\fi MR }
\providecommand{\MRhref}[2]{%
  \href{http://www.ams.org/mathscinet-getitem?mr=#1}{#2}
}
\providecommand{\href}[2]{#2}

\end{document}